\documentclass[12pt]{amsart}

\usepackage{amsmath,amstext,amssymb,amsopn,amsthm}
\usepackage{url,verbatim}
\usepackage{mathtools}
\usepackage{enumerate}

\usepackage{color,graphicx}

\usepackage[margin=30mm]{geometry}
\usepackage{eucal,mathrsfs,dsfont}

\numberwithin{equation}{section}

\allowdisplaybreaks

\usepackage[utf8]{inputenc}
\usepackage{graphicx,amsmath,amssymb,enumerate,color,amsthm}
\usepackage{hyperref}
\usepackage{epstopdf}
\renewcommand{\P}{\mathbb{P}}

\newcommand{\IGFT}{{\rm IED}}
\newcommand{\IED}{{\rm IED}}

\newcommand{\F}{\mathcal{F}}
\newcommand{\essinf}{\mathrm{ess\,inf}}

\newcommand{\G}{\mathcal{G}}

\newcommand{\R}{\mathbb{R}}

\newcommand{\Z}{\mathbb{Z}}
\newcommand{\E}{\mathbb{E}}
\newcommand{\1}{\mathbf{1}}

\newtheorem{thm}{Theorem}[section]
\newtheorem{lemma}[thm]{Lemma}
\newtheorem{corol}[thm]{Corollary}
\newtheorem{propo}[thm]{Proposition}

\theoremstyle{definition}

\newtheorem{defin}[thm]{Definition}

\newtheorem{remark}[thm]{Remark}
\newtheorem{example}[thm]{Example}

\newcommand{\ldm}{LDM}
\newcommand{\eps}{\varepsilon}

\newcommand{\wt}{\widetilde}

\newcommand{\cY}{\mathcal{Y}}

\title[Stochastic fixed point equation]{Stochastic fixed point equation and local dependence measure}

\author{Krzysztof Burdzy \and Bartosz Ko\L{}odziejek \and Tvrtko Tadi\'c}

\address{KB: Department of Mathematics, Box 354350, University of Washington, Seattle, WA 98195, USA}
\email{burdzy@uw.edu}

\address{BK: Faculty of Mathematics and Information Science, Warsaw University of Technology, Koszykowa 75, 00-662 Warsaw, Poland}
\email{b.kolodziejek@mini.pw.edu.pl}

\address{TT: Microsoft Corporation, City Center Plaza Bellevue, One Microsoft Way, Redmond, WA 98052, USA}
\email{tvrtko@math.hr}

\keywords{Tail estimates, stochastic fixed point equation, iterated random sequence, Fleming-Viot type process}
\subjclass[2010]{60G10, 37M10, 60J05}
\begin{document}

\maketitle

\begin{abstract}
We study solutions to the stochastic fixed point equation $X\stackrel{d}{=}AX+B$ where the coefficients $A$ and $B$ are nonnegative random variables. We introduce the ``local dependence measure'' (\ldm) and its Legendre-type transform to analyze the left tail behavior of the distribution of $X$. We discuss the relationship of \ldm{} with earlier results on the stochastic fixed point equation and we apply \ldm\  to prove a theorem on a Fleming-Viot-type process.
\end{abstract}

\section{Introduction}

Our research on the stochastic fixed point equation is motivated by a problem arising in the theory of the so-called Fleming-Viot processes. We made a partial progress towards our goal in \cite{IEDPaper}. This article contains new ideas that lead to the complete solution of that problem; see Section \ref{sec:FlemingViot} for details. Needless to say, we hope that the new technique developed in this paper will have applications beyond the theory of Fleming-Viot processes.

Given a pair of random variables $(A,B)$, an independent random variable $X$ is said to satisfy the stochastic fixed point equation 
if
\begin{equation}\label{stochasticFixedPointEquation}
X\stackrel{d}{=}AX+B. 
\end{equation}
The behavior of the solution, especially the left and right tails, has been extensively studied. A classical result (\cite{Kesten,GoldieAAP}) 
says that under some assumptions on $(A,B)$, for some $\alpha, C_-, C_+ >0$,
\begin{equation}\label{tailEstimates}
\P(X>x)\sim C_+x^{-\alpha}\quad \textrm{and}\quad \P(X<-x)\sim C_-x^{-\alpha}, 
\end{equation}
as $x\to\infty$ (see Theorem \ref{thm:KestenGoldieResult} for a fully rigorous version). An excellent review of  the subject can be found in  \cite{X=AX+B}. 

It can be shown that  
if $A$ and $B$ are nonnegative random variables then a nonconstant solution $X$ to \eqref{stochasticFixedPointEquation} must be also a nonnegative random variable (we do not present a proof because this claim is not needed for the main application of \eqref{stochasticFixedPointEquation} in Section \ref{sec:FlemingViot}).
If $X$ is nonnegative then the first estimate in $(\ref{tailEstimates})$ is still meaningful and informative, but the second one is not because for $x>0$ we have 
$\P(X<-x)=0$. In this article, we will continue the analysis of  the behavior of $\P(X<x)$ as $x\to 0^+$ initiated in \cite{IEDPaper}. 

We will introduce a new concept of ``local dependence measure'' (\ldm) and its Legendre-type transform. We will relate \ldm\  to concepts discussed in  \cite{IEDPaper}: inverse exponential decay of the tail of $B$, and positive quadrant dependence
of $A$ and $B$. We will illustrate the power of \ldm \ by a few examples, including the proof of a result on the Fleming-Viot model.

\subsection{Organization of the paper}
Section \ref{sec:generalResults1} is devoted to the basic general properties of 
solutions to the stochastic fixed point equation \eqref{stochasticFixedPointEquation}. 
We 
recall the conditions that guarantee the existence and uniqueness of the solution in Theorem \ref{m8.1} and Corollary \ref{thm:uniquenessAndConvergence}.

In Section \ref{inverseExponentialPotential} we define the local dependence measure (\ldm) for the random variables $(A,B)$ in \eqref{stochasticFixedPointEquation}, and its Legendre-type transform.
We study basic properties of these functions and present their first application to the stochastic fixed point equation.

In Section \ref{sec:perpetuity} we show that if  \ldm\ for the random variables $(A,B)$ exists then the solution to \eqref{stochasticFixedPointEquation} is a random variable with an ``inverse exponential decay'' left tail.

Section \ref{sec:examples} is devoted to calculating explicit formulas for \ldm \ (Proposition \ref{y31.2}) and its Legendre-type transform (Proposition \ref{propo:phiABPQD})  when $A$ and $B$
are positively quadrant dependent random variables. 

In Section \ref{sec:logarithmicLowerEnvelope} we prove that if $X_n=A_n X_{n-1}+B_n$ for $n\geq1$, $X_0=0$ and $(A_n,B_n)_{n\geq1}$ is a sequence of independent copies of $(A,B)$, then
$$\liminf_{n\to \infty}\frac{X_n}{H^{-1}(\log n)}= (\lambda^\ast)^{1/\rho}, $$
where $H$ is a regularly varying function introduced in the definition of \ldm, $\lambda^\ast$ is the fixed point for the Legendre-type transform, and $\rho$ is a parameter in the definition of \ldm.

In Section \ref{sec:FlemingViot} we 
apply an \ldm \ to prove a version of the Law of Iterated Logarithm for a Fleming-Viot type process.

\section{General results on stochastic fixed point equation}\label{sec:generalResults1}

In this section we will introduce notation and conventions used in the rest of the paper, and present some known general results, with references but no proofs.

In this section, and this section only, we will allow the coefficients $A$ and $B$ of the stochastic fixed point equation
\begin{equation}\label{y20.1}
X\stackrel{d}{=}AX+B,
\end{equation}
to take arbitrary (positive and negative) values. Starting with Section \ref{inverseExponentialPotential}, we will assume that $A,B \geq 0$, a.s.

 We will say that $X$, a random variable with values in $\R$, is a solution to \eqref{y20.1} if one can construct $X,A$ and $B$ on the same probability space in such a way that $X$ is independent of $(A,B)$ and \eqref{y20.1} is satisfied.
 
We will always use $(A_n,B_n)$ to denote a vector with the same distribution as $(A,B)$ (the distribution of $(A,B)$ can change from one context to another). 

Let $(A_n,B_n)_{n\geq 1}$ be an i.i.d. sequence  and define random affine maps from $\R$ to itself by
\[
\Psi_n(t)= A_n\,t +B_n,\qquad t\in\R.
\]
Clearly, $(\Psi_n)$ is an i.i.d. sequence. 
Suppose that $X_0$ is independent from $(A_n,B_n)_{n\geq 1}$ and
let
\begin{align}\label{y20.2}
X_{n} = \Psi_n(X_{n-1}) = A_n X_{n-1}+B_n,
\end{align}
for $n\geq 1$. 
Note that $(X_n)$ is a Markov chain.
It is easy to check that
\[
X_n = \left(\sum_{k=1}^n B_k  \prod_{j=k+1}^{n} A_j\right) +  X_0  \prod_{i=1}^{n} A_i.
\]
We define another sequence of affine mappings, starting with $S_0(t)=t$ for all $t\in\R$, and continuing inductively by
\[
S_n(t) = S_{n-1}\circ \Psi_n(t) = S_{n-1}(A_n t + B_n),
\]
for $n\geq 1$.
Then we have 
\[
S_n(t) = \left(\sum_{k=1}^n B_k \prod_{j=1}^{k-1} A_j\right) + t  \prod_{i=1}^{n} A_i,
\]
with the convention that $\displaystyle \prod_{j=k}^{m}A_j = 1$ if $m<k$. 
Re-indexing of the sequence $(A_n,B_n)_{n\geq 1}$ easily shows that 
\begin{align}\label{y21.4}
X_n \stackrel{d}{=} S_n(X_0)
\end{align}
for each $n\geq 1$. 

The following follows from the ``Principle'' stated on page 264 of \cite{Letac}.

\begin{lemma}\label{m8.2}
If for each $t\in\R$ the sequence $(S_n(t))$ converges almost surely to a limit, say $S$, which does not depend on $t$, then the law of $S$ is the unique solution to \eqref{y20.1}. Moreover, $(X_n)$ converges to $S$ in distribution, for any $X_0$. 
\end{lemma}

The only natural candidate for the limit $S$ is the series
\begin{align}\label{eq:defS}
S := \sum_{k=1}^\infty B_k \prod_{j=1}^{k-1} A_j.
\end{align}
If $\P(A=0)>0$, then $N=\inf\{n \geq 1 \colon A_n=0\}$ is a.s. finite and
\[
S_n(t) = \sum_{k=1}^N B_k \prod_{j=1}^{k-1} A_j
\]
for all $n\geq N$. Thus, the condition $\P(A=0)>0$ ensures the a.s. convergence of $(S_n(t))$ for all $t\in\R$. 

For $x>0$, let $f_A(x)=\int_0^x \P(|A|<e^{-t})dt$. The following theorem characterizes  almost sure convergence of $(S_n(X_0))$. It follows from a more general result in \cite[Thm. 2.1]{GM00}.
\begin{thm}\label{m8.1}
Suppose that $\P(B=0)<1$ and $\P(A=0)=0$. Then,
\begin{equation}
\label{eq:c1}
\sum_{n=1}^\infty |B_n| \prod_{j=1}^{n-1} |A_j| < \infty,\quad  a.s. 
\end{equation}
is equivalent to 
\begin{align}\label{eq:c2}
\prod_{j=1}^{n} A_j\to 0\,\, (n\to\infty)
\qquad\mbox{ and } \qquad
\int_{(1,\infty)} \frac{\log b}{f_A(\log b)} \P_{|B|}(db)<\infty.
\end{align}

Each of the above equivalent conditions \eqref{eq:c1} and \eqref{eq:c2} implies that, a.s.,
\begin{align}\label{m8.3}
S_n(X_0) \to S,\qquad  n\to\infty.
\end{align}
Conversely, if
\begin{align}\label{eq:c3}
\P(A c+B=c )<1\,\,\mbox{ for all }c\in\R,
\end{align}
and \eqref{eq:c2} does not hold, then 
\[
|S_n(X_0)| \stackrel{\P}{\longrightarrow} \infty,\qquad n\to\infty.
\]
\end{thm}

According to \cite[Cor. 4.1]{GM00} or \cite[Theorem 2.1.3]{X=AX+B}, a sufficient condition for \eqref{eq:c2}  is
\begin{align}\label{m8.4}
\E[\log|A|]<0\qquad \mbox{ and }\qquad \E[ \log^+|B|]<\infty.
\end{align}

If \eqref{eq:c3} does not hold, i.e., there exists $c\in\R$ such that $A c+B=c$, a.s., then $X\equiv c$ is the unique solution to \eqref{y20.1}.

The following result follows from Lemma \ref{m8.2} and Theorem \ref{m8.2}; or from  \cite[Theorem 3.1]{GM00}.
\begin{corol}\label{thm:uniquenessAndConvergence}
	Assume that the nondegeneracy condition \eqref{eq:c3} is satisfied.
		
	(i)
	If $\P(A=0)>0$ or \eqref{eq:c2} holds, then for every $X_0$,
	\[
	X_n\stackrel{d}{\longrightarrow} X,\qquad n\to\infty,
	\]
	where $X$ is the unique solution to \eqref{y20.1}.	
	
	(ii) If $\P(A=0)=0$ and \eqref{eq:c2} fails, then for every $X_0$,
	\[
	|X_n|\stackrel{\P}{\longrightarrow} \infty,\qquad n\to\infty.
	\]
\end{corol}

	We say that a real random variable $Y$ is stochastically majorized by  $Z$, and we write $Y\leq_{st}Z$,  if $\P(Y\leq x)\geq \P(Z\leq x)$ for all $x\in\R$. 

\begin{lemma}\label{pro:p1}
	Consider $(X_n)$  defined in \eqref{y20.2}.  
If $A\geq0$, a.s., and $X_1\geq_{st}X_0$, then for all $n \geq 1 $,
\[
X\geq_{st} X_{n+1}\geq_{st}X_n\geq_{st} X_0.
\]
\end{lemma}
\begin{proof}
	It is enough to show that  $X_{n+1}\geq_{st} X_n$ for all $n \geq 1 $. 
	We proceed by induction. Suppose that $X_n\geq_{st} X_{n-1}$. Since $A_{n+1}\geq0$ a.s., we have
	\[
	X_{n+1} = A_{n+1} X_n+B_{n+1}\geq_{st} A_{n+1} X_{n-1}+B_{n+1} \stackrel{d}{=} X_{n}.
	\]
\end{proof}

If both $A$ and $B$ are nonnegative and $X_0=0$ then $X_1=B_1\geq0$ and, therefore, the assumptions of Lemma \ref{pro:p1} are satisfied. In this case, for all $x\geq 0$ and $n\geq 1$,
\begin{align}\label{eq:st}
\P(X\leq x)\leq \P(X_n\leq x).
\end{align}

\section{Local dependence measure and Legendre-type transformation}\label{inverseExponentialPotential}

From now on, we will assume that the coefficients $A$ and $B$ of the stochastic fixed point equation \eqref{y20.1} are nonnegative, i.e., $A,B \geq 0$, a.s.

The concept of a regularly varying function is well known. For the definition and a review of properties of regularly varying function needed in this project, see  \cite[Sec. 2]{IEDPaper}. 

\begin{defin}\label{definition:IEDRandomVariable} (\cite{IEDPaper})
We say that a nonnegative random variable $X$ has an inverse exponential decay of the left tail with degree $\rho>0$ if
\begin{equation}
\lim_{x\to 0^+}\frac{-\log \P(X<x)}{H(x)}=\lambda,\label{eq:definitionOfLambda0}
\end{equation}
for a regularly varying function $H$ with index $-\rho$ at zero and $\lambda\in[0,\infty]$.
We  call such a random variable \IGFT$^\rho_H(\lambda)$-random variable.
\end{defin}

Sometimes we will write $f(x) \sim g(x)$, $x\to \infty$, to indicate that $\lim_{x\to \infty} f(x) / g(x) =1$ (the same notation will apply in the case when $x$ goes to a different limit).

\begin{remark}\label{remark:monotoneEquivalents}
We will argue that if $H$ is regularly varying with index $-\rho<0$ at $0$ then there
exists  a continuous nonincreasing regularly varying function $\wt H$ with index $-\rho<0$ at $0$, such that 
$$\lim_{x\to 0^+} H(x)/\widetilde{H}(x)=1.$$
To see this, first apply the
Smooth Variation Theorem (\cite[Theorem 1.8.2]{regularVariation}) which states that for any regularly varying function $f$ there exists a smooth function $f_1$ with $f(x)\sim f_1(x)$. So, we have $H(x)\sim H_1(x)$ for some smooth $H_1$.

Then, the monotone equivalent $\wt H$ to $H_1$ can be constructed using \cite[Theorem 1.5.3]{regularVariation}. By continuity of $H_1$, the function $\wt H$ will be also  continuous. For
more details see \cite[Thm. 1.5.4]{regularVariation}
or \cite[Cor. 4.2]{BKS}. Without loss of generality, we will assume from now on that every regularly varying function $H$ is continuous and monotone. Then $H$ has an inverse, denoted by $H^{-1}$, which is regularly varying with index $-1/\rho$ at $\infty$.
\end{remark}

The ultimate goal of this project is to develop an effective tool for the analysis of the lower tail of the solution to \eqref{y20.1}. The random variables $A$ and $B$ in that formula  are not necessarily independent.
We will quantify their dependence using the ``local dependence measure'' (\ldm) defined below.

\defin{We will say that, for a pair of nonnegative random variables $(A,B)$, a function $g:[0,\infty)\to [0,\infty]$ is their $(\rho,H)$-local dependence measure ($(\rho,H)$-\ldm) if
for a regularly varying function $H$ with index $-\rho<0$ at $0$, 
\begin{equation}
g(y)=\lim_{\varepsilon\to 0^+} \frac{-\log\P(\varepsilon Ay+B<\varepsilon)}{H(\varepsilon)}.
\label{eq:g(y)AsLimit}
\end{equation}
} 

\begin{remark}
If $g(0)>0$, then \eqref{eq:g(y)AsLimit} implies
\begin{equation*}
\lim_{\varepsilon\to 0^+} \frac{\log\P(\varepsilon Ay+B<\varepsilon)}{\log\P(B<\varepsilon)} = \frac{g(y)}{g(0)},\qquad y\geq 0.
\end{equation*}
Similar conditions for the distribution of a pair $(A,B)$ were considered in literature in related context. In particular, in  \cite[Thm. 2.1]{Gamma} it is assumed that there exists a finite function $f$ such that
\begin{align}\label{eq:alienf}
\lim_{x\to\infty}\frac{\P(Ay+B>x)}{\P(B>x)}=f(y),\qquad y\in\R.
\end{align}
If $A$ and $B$ above are independent, $A$ has a finite moment generating function and $x\mapsto\P(e^B>x)$ is regularly varying with index $-\alpha\leq0$ at $\infty$, then by the Breiman lemma (see \cite{Breiman}) we have $f(y)=\E[e^{\alpha Ay}]$. However, if $A$ and $B$ are not independent, yet \eqref{eq:alienf} holds, then $f$ may be of different form (see \cite[Remark 2.3]{Gamma}). 
Another condition of similar nature  was stated in \cite[(9)]{Palm}.
\end{remark}

In Section \ref{sec:examples} we will show that if random variables $A$ and $B$ are positively quadrant dependent, then $g$ defined in \eqref{eq:g(y)AsLimit} can be given explicitly. If  $A$ and $B$ are not positively quadrant dependent then the form of $g$ may vary significantly (see Example \ref{g(0)neqg(0^+)} and Proposition \ref{propo:inverseDependencyPotentialCalculation}).

\begin{lemma}\label{j26.1}
If $g$ is $(\rho, H)$-\ldm{}, then $g:[0,\infty)\to [0,\infty]$ is a nondecreasing function.
\end{lemma}
\begin{proof}
If $y_1\leq y_2$ then $\P(\varepsilon Ay_2+B<\varepsilon)\leq \P(\varepsilon Ay_1+B<\varepsilon)$. This and $(\ref{eq:g(y)AsLimit})$ imply that $g(y_1)\leq g(y_2)$.
\end{proof}

\begin{remark}\label{j27.1}
If $(\ref{eq:g(y)AsLimit})$ holds for $y=0$ then $B$ is an
\IGFT$^\rho_H(g(0))$-random variable.
\end{remark}
\begin{example}\label{g(0)neqg(0^+)}
Some of our results hold only if the \ldm{} $g$ is continuous at $0$. 
The question of whether every \ldm{} must be continuous at $0$ does not seem to be trivial.  Therefore, we present an example showing that  $g$ can be discontinuous at $0$. Here we do not make any claims concerning continuity of $g$ on $(0,\infty)$.

Let $A=V/U$ and $B=U$, where $V$ and $U$ are positive continuous random variables such that 
$\P(V< v) = e^{-1/v}$ for $v>0$, and 
$$\P(U\in du, V\in dv)=\left(c_1e^{-\lambda_1/u}\1_{(v<1)}+c_2 e^{-\lambda_2/u}\1_{(v\geq 1)}\right)\1_{(u\leq 1)}du\P(V\in dv),$$
where $\lambda_1>\lambda_2 >0$, and  $c_1$ and $c_2$ are positive normalizing constants.

For $\eps >0$, 
\begin{align}\label{m3.1}
  \P(\varepsilon Ay +B<\varepsilon) 
  &=\P\left(\varepsilon (V/U)y +U<\varepsilon \right)
= \P\left( V<\frac{U(\varepsilon-U) }{\varepsilon y} \right)\\
&=\P\left( V<\frac{U(\varepsilon-U) }{\varepsilon y}, U\in (0,\varepsilon) \right). \notag
\end{align}
For fixed $y>0$ and small enough $\varepsilon>0$ we have $u(\varepsilon-u) /( \varepsilon y) <\varepsilon/y<1$ for $u\in (0,\varepsilon)$. 
Hence we obtain from \eqref{m3.1}, using the substitution $u=\eps t$,
 \begin{align*}
 \P&(\varepsilon Ay +B<\varepsilon)  =\int_0^\varepsilon\int_0^{u(\varepsilon-u)/\varepsilon y}  c_1e^{-\lambda_1/u}\P(dv)du =\int_0^\varepsilon  c_1e^{-\lambda_1/u}\P\left(V< \frac{u(\varepsilon-u)}{\varepsilon y}\right)du\\
& =\varepsilon\int_0^1  c_1e^{-\lambda_1/\varepsilon t}\P\left( V< \frac{\varepsilon t(1-t)}{y}\right)dt
=\varepsilon\int_0^1  c_1\exp\left(-\frac{\lambda_1}{\varepsilon t}\right)\exp\left(-\frac{y}{\varepsilon t(1-t)}\right)dt\\
&=\varepsilon c_1 \int_0^1 \exp\left(-\frac{\lambda_1+y}{\varepsilon t}-\frac{y}{\varepsilon (1- t)}\right)dt.
 \end{align*}
This and Lemma \ref{lemma:minimumExponent} imply that
\begin{align*}
g(y)&=-\lim_{\varepsilon \to 0^+} \varepsilon\log \P(\varepsilon Ay +B<\varepsilon) \\
&=-\lim_{\varepsilon \to 0^+} \varepsilon\log (\eps c_1)
-\lim_{\varepsilon \to 0^+} \varepsilon\log
\left(
\int_0^1 \exp\left(-\frac{\lambda_1+y}{\varepsilon t}-\frac{y}{\varepsilon (1- t)}\right)dt\right)\\
&= (\sqrt{\lambda_1+y} + \sqrt{y})^2.
\end{align*}
Hence $g(0^+) =  \lambda_1$. Since $\lambda_2<\lambda_1$, by Lemmas \ref{lemma:minimumExponent} and \ref{lemma:differenceBetweenIEDFunctions},
\begin{align*}
 g(0) &= -\lim_{\varepsilon \to 0^+} \varepsilon\log \P(B<\varepsilon) =-\lim_{\varepsilon \to 0^+} \varepsilon\log \P(U<\varepsilon)\\
&=-\lim_{\varepsilon \to 0^+} \varepsilon \log\left( \int_0^{\varepsilon} \left(c_1\P(V<1)e^{-\lambda_1/u}+c_2\P(V\geq 1)e^{-\lambda_2/u}
\right)du\right)\\
&= \lambda_2 < \lambda_1 = g(0^+).
\end{align*}

\end{example}

\rm

\begin{defin}\label{j17.2}
For a function $g:[0,\infty)\to [0,\infty]$, we let
\begin{equation}\label{eq:PhiDefinition}
 \phi_\rho(\lambda) =\inf_{y>0}\left\{g(y)+\frac{\lambda}{y^\rho}\right\},
 \qquad \lambda \geq 0.
\end{equation}
\end{defin}

The  Legendre-type transform $\phi_\rho(\lambda)$
will play a key role in our analysis. We will illustrate its significance with a couple of results, before deriving its basic properties.

\begin{thm}\label{thm:inverseDependencyPotentialMain}
Suppose that $g:[0,\infty)\to [0,\infty]$ 
is the $(\rho,H)$-\ldm\ for  random variables $(A,B)$, and let $X$ be an independent 
\IGFT$^\rho_H(\lambda)$-random variable. 
If $g(0^+)=g(0)$ or $\lambda>0$ then $AX+B$ is an $\IED^{\rho}_H(\phi_\rho(\lambda))$-random variable. 
\end{thm}
\begin{proof}
First we will show that
\begin{equation}
\label{eq:upperBoundInequality}
\limsup_{\varepsilon\to 0^+}\frac{- \log\P(AX+B<\varepsilon)}{H(\varepsilon)} \leq \inf_{y> 0}\left\{g(y)+\frac{\lambda}{y^\rho}\right\}
=\phi_\rho(\lambda).
\end{equation}
For any $y>0$,
\begin{align*}
 \P(AX+B<\varepsilon)&\geq  \P(AX+B<\varepsilon, X<\varepsilon y)\geq \P(\varepsilon Ay +B<\varepsilon, X<\varepsilon y)\\
& =\P(\varepsilon Ay +B<\varepsilon)\P(X<\varepsilon y).
\end{align*}
This implies that
\begin{align*}
&\limsup_{\varepsilon\to 0^+}\frac{-\log\P(AX+B<\varepsilon)}{H(\varepsilon)}\\
&\leq \lim_{\varepsilon\to 0^+}\frac{-\log\P(\varepsilon Ay +B<\varepsilon)}{H(\varepsilon)}+\lim_{\varepsilon\to 0^+}\frac{-\log \P(X<\varepsilon y)}{H(\varepsilon y)}\frac{ H(\varepsilon y)}{H(\varepsilon)}\\
&\leq  g(y)+\frac{\lambda}{y^\rho}. 
\end{align*}
Since $y$ is an arbitrary number in $(0,\infty)$, we obtain \eqref{eq:upperBoundInequality}. 

We will consider three cases: (i) $\lambda =0$, (ii) $\lambda >0$ and 
$\phi_{\rho}(\lambda)<\infty$, and (iii) $\lambda >0$ and 
$\phi_{\rho}(\lambda)=\infty$.

(i) Consider  $\lambda =0$. By Lemma  \ref{lemma:concavePhi} (d) (proved below) and the assumption that $g(0^+)=g(0)$,
$$\phi_{\rho}(0)=g(0)=\lim_{\varepsilon\to 0^+}
\frac{-\log \P(B<\varepsilon)}{H(\varepsilon)}\leq \liminf_{\varepsilon\to 0^+}\frac{-\log\P(AX+B<\varepsilon)}{H(\varepsilon)}.$$
This and \eqref{eq:upperBoundInequality} prove the theorem in the case $\lambda=0$.

(ii) Under the assumption that $\lambda>0$ and $\phi_{\rho}(\lambda)<\infty$, there exists $a>0$ such that 
$\lambda/a^{\rho}\geq \phi_{\rho}(\lambda)$. Hence 
$$\liminf_{\varepsilon\to 0^+}\frac{-\log \P(AX+B<\varepsilon,X<\varepsilon a)}{H(\varepsilon)}\geq \liminf_{\varepsilon\to 0^+}\frac{-\log \P(X<\varepsilon a)}{H(\varepsilon)}=\frac{\lambda}{a^{\rho}}\geq \phi_{\rho}(\lambda).$$

Since, by Lemma \ref{j26.1}, $g$ in nondecreasing, we have $\sup_{y>0}g(y)\geq \phi_{\rho}(\lambda)$.
Therefore, for $\delta>0$ there exists $b>0$ such that $g(b)\geq \phi_{\rho}(\lambda)-\delta$.
Thus,
\begin{align*}
\liminf_{\varepsilon\to 0^+}\frac{-\log \P(AX+B<\varepsilon,X\geq \varepsilon b)}{H(\varepsilon)}
&\geq \liminf_{\varepsilon\to 0^+}\frac{-\log \P(\varepsilon Ab+B<\varepsilon)}{H(\varepsilon)}\\
&=g(b)\geq \phi_{\rho}(\lambda)-\delta.
\end{align*}
We  conclude that for $\eta>0$ and small $\varepsilon>0$
\begin{align}\label{eq:lowerTailBound}\begin{split}
 \P(AX+B<\varepsilon, X<\varepsilon a)&\leq \exp(-H(\varepsilon)(\phi_{\rho}(\lambda)-\eta)),\\
 \P(AX+B<\varepsilon, X\geq \varepsilon b)&\leq \exp(-H(\varepsilon)(\phi_{\rho}(\lambda)-\eta)).
\end{split}\end{align}
 
For $y,h>0$ and small $\varepsilon>0$, 
\begin{align*}
 \P&(AX+B<\varepsilon, \varepsilon y\leq X<\varepsilon (y+h))\leq \P(\varepsilon Ay+B<\varepsilon)\P(X<\varepsilon (y+h))\\
&\leq 
\exp(-H(\varepsilon)(g(y)-\eta))
\exp(-H(\varepsilon)(\lambda/(y+h)^\rho-\eta))\\
&\leq \exp(-H(\varepsilon)(g(y)+\lambda/y^\rho-\lambda/y^\rho+\lambda/(y+h)^\rho-2\eta)).
\end{align*}
For $y>0$, by definition, $\phi_\rho(\lambda)\leq g(y)+{\lambda}/{y^{\rho}}$. We  have 
$1/y^{\rho}-1/(y+h)^{\rho}\leq h\rho/y^{\rho+1}$.
Hence,
 $$\P(AX+B<\varepsilon, \varepsilon y\leq X<\varepsilon (y+h)) \leq 
 \exp(-H(\varepsilon)(\phi_{\rho}(\lambda)-\lambda h\rho/y^{\rho+1}-2\eta)). $$
From this we obtain
\begin{align}
& \P(AX+B<\varepsilon, \varepsilon a\leq X< \varepsilon b) \nonumber\\
&=\sum_{k=1}^n \P(AX+B<\varepsilon, \varepsilon (a+h_{k-1})\leq X< \varepsilon (b+h_k))\nonumber\\
&\leq\sum_{k=1}^n 
\exp(-H(\varepsilon)(\phi_{\rho}(\lambda)-\lambda h\rho/(a+h_{k-1})^{\rho+1}-2\eta))\nonumber\\
&\leq n 
\exp(-H(\varepsilon)(\phi_{\rho}(\lambda)-\lambda h\rho/a^{\rho+1}-2\eta)),\label{eqInterBound}
\end{align}
where $h_0=0$, and $h_k-h_{k-1}=(b-a)/n$ for $k=1,\ldots,n$.
Using \eqref{eq:lowerTailBound} and \eqref{eqInterBound} we get 
$$\P(AX+B<\varepsilon)\leq (n+2)
\exp(-H(\varepsilon)(\phi_{\rho}(\lambda)-\lambda h\rho/a^{\rho+1}-2\eta)).$$
Hence 
$$\liminf_{\varepsilon\to 0^+}\frac{-\log \P(AX+B<\varepsilon)}{H(\varepsilon)}\geq \phi_{\rho}(\lambda)-\lambda h\rho/a^{\rho+1}-2\eta.$$
By first letting $\eta \downarrow 0$ and then $n\uparrow\infty$ (so that $h\downarrow 0$), 
we get 
$$\liminf_{\varepsilon\to 0^+}\frac{-\log \P(AX+B<\varepsilon)}{H(\varepsilon)}\geq \phi_{\rho}(\lambda).$$
This and \eqref{eq:upperBoundInequality} prove the theorem in this case.

(iii) If $\phi_{\rho}(\lambda)=\infty$, then $g(y)=\infty$ for all $y>0$. We have
\begin{align*}
\P&(AX+B<\eps)  = \P(AX+B<\eps,X<\eps y)+\P(AX+B<\eps,X\geq \eps y) \\
& \leq \P(X<\eps y)+\P(\eps A y +B<\eps) \leq 2 \max\{\P(X<\eps y), \P(\eps A y +B<\eps) \}.
\end{align*}
Thus,
\begin{align*}
\frac{-\log\P(AX+B<\eps)}{H(\eps)}\geq \frac{-\log 2}{H(\eps)}+ \min\left\{ \frac{-\log\P(X<\eps y)}{H(\eps)}, \frac{-\log\P(\eps A y +B<\eps)}{H(\eps)} \right\}.
\end{align*}
The right hand side  converges to $\min\{ \lambda/y^\rho, g(y) \}=\lambda/y^\rho$ when $\eps \to 0^+$. We can  make $\lambda/y^\rho$ arbitrarily large by choosing $y$ small enough. This shows that  $AX+B$ is an $\IED^{\rho}_H(\infty)$-random variable. Since $\phi_{\rho}(\lambda)=\infty$, this completes the proof.
\end{proof}

\begin{remark}
The condition $g(0)=g(0^+)$ is  necessary for the statement of  Theorem \ref{thm:inverseDependencyPotentialMain} to hold for $\lambda=0$. To see this, note that if
$X\equiv 0$ then $X$ is an $\IED^\rho_H(0)$ random variable,  $AX+B=B$ and $AX+B=B$ is  $\IED^\rho_H(g(0))$.
However, if we pick $A$ and $B$ as in Example \ref{g(0)neqg(0^+)} then $\phi_\rho(0)=g(0^+)>g(0)$, by Lemma \ref{lemma:concavePhi} (d).   
\end{remark}

\begin{corol}\label{corol:fixedPoint}
Suppose that $g:[0,\infty)\to [0,\infty]$ 
is the $(\rho,H)$-\ldm\ for  random variables $(A,B)$.
 If $\lambda\in(0,\infty)$ and $X$ is an $\IED^\rho_H(\lambda)$-random variable that is a solution
to \eqref{y20.1} then  $\lambda = \phi_\rho(\lambda)$. 
\end{corol}
\begin{proof}
The claim follows from Theorem \ref{thm:inverseDependencyPotentialMain}.
\end{proof}
\rm
We will now investigate  basic properties of $\phi_\rho$.

\begin{lemma}\label{lemma:concavePhi}
The function  $\phi_\rho:[0,\infty)\to [0,\infty]$ 
defined in \eqref{eq:PhiDefinition}
is  nondecreasing and concave. Moreover,
\begin{enumerate}[(a)]
 \item If there exists $y_0>0$ such that $g(y_0)<\infty$ then $\phi_\rho(\lambda)<\infty$ for all $\lambda \geq 0$.
 \item If $g$ is  bounded by $M$ then $\phi_\rho$ is also  bounded by $M$.
 \item If $\phi_\rho (0)>0$ then $\phi_\rho$ has at most one fixed point, i.e., $\phi_\rho(\lambda) = \lambda$ for at most one $\lambda$.
 \item We have $\phi_\rho(0)= g(0^+)\geq g(0)$.
 \item 
If there exists $\lambda_0\geq 0$ such that $\phi_{\rho}(\lambda_0)<\infty$, then $\phi_\rho(\lambda)<\infty$ for all $\lambda\geq 0$.
\end{enumerate}
\end{lemma}
\begin{proof}
It follows directly from the definition \eqref{eq:PhiDefinition} that $\phi_\rho$ is nondecreasing. Moreover, as the infimum of a family of affine functions, $\phi_\rho$ is concave.

(a)  Note that $\phi_\rho(\lambda)\leq g(y_0)+\lambda y_0^{-\rho}$ for $\lambda\geq 0$.

(b) Since $\sup_{y>0}g(y)\leq M$, the definition \eqref{eq:PhiDefinition}
shows that $\phi_\rho(\lambda)\leq M+\lambda y^{-\rho}$ for every $y>0$. The claim follows by letting $y\to \infty$ in \eqref{eq:PhiDefinition}.

 (c) Suppose that $\phi_\rho (0)>0$ and $\lambda_2 > \lambda_1>0$ are fixed points. Then $\lambda_1$ is a convex combination 
of $0$ and $\lambda_2$, i.e., there exists $\alpha \in (0,1)$ such that $\lambda_1=\alpha \lambda_2+(1-\alpha)0$.
Since $\phi_\rho$ is concave,
$$\phi_\rho(\lambda_1)\geq \alpha \phi_\rho(\lambda_2)+(1-\alpha)\phi_\rho(0)>\alpha \lambda_2+(1-\alpha)0=\lambda_1.$$
Hence $\lambda_1$ is not a fixed point. This contradiction proves the claim.

(d)
This follows from Lemma \ref{j26.1} and \eqref{eq:PhiDefinition}.

(e) 
If there exists $\lambda_0\geq 0$ such that $\phi_{\rho}(\lambda_0)<\infty$ then $g(y_0) < \infty$ for some $y_0 >0$. This and the definition \eqref{eq:PhiDefinition} imply that $\phi_\rho(\lambda)<\infty$ for all $\lambda\geq 0$.

\end{proof}\rm

\begin{lemma}\label{y31.1}
If  $\phi_\rho$ is finite then it is  continuous.
\end{lemma}
\begin{proof}
By Lemma \ref{lemma:concavePhi}, $\phi_\rho$ is a concave function. A classical result in (convex) analysis says that a finitely 
valued concave function is continuous.
\end{proof}

\begin{defin}\label{j17.1}
	Let  
	\[
	\lambda^\ast = \inf_{y> 1}\left\{\frac{y^\rho}{y^\rho-1}g(y)\right\}.
	\]
\end{defin}

\begin{lemma}\label{pro:lambdastar}\ 

(i) Suppose $c\geq0$. Then $\phi_{\rho}(c)\geq c$ if and only if $c\leq \lambda^\ast$.

(ii) If $\lambda^\ast<\infty$ then $\phi_\rho\left(\lambda^\ast\right)= \lambda^\ast$.
\end{lemma}
\begin{proof}
(i) 
We have $\phi_\rho(c)=\inf_{y>0}\{g(y)+c y^{-\rho}\}\geq c$ if and only if
\[
g(y)+c y^{-\rho} \geq  c
\]
for all $y>0$. The above inequality is always satisfied for $y\leq 1$. Thus,  $\phi_\rho(c)\geq c$ if and only if
\[
\frac{y^\rho}{y^\rho-1}g(y) \geq c 
\]
for all $y>1$, which is equivalent to $c\leq\lambda^\ast$.

(ii)  By (i), we have
\begin{align}\label{m10.2}
\{c\geq 0: \phi_\rho(c)\geq c\} = [0,\lambda^\ast].
\end{align}
 Thus, if $\lambda^\ast<\infty$, for any sequence $\lambda_n\downarrow \lambda^\ast$, we have
   $\phi_\rho(\lambda_n) < \lambda_n$.
This and the continuity of $\phi_\rho$ imply that $\phi_\rho\left(\lambda^\ast\right)\leq \lambda^\ast$. 
This inequality and part (i) applied to $c=\lambda^\ast$ yield (ii).
\end{proof}

\begin{propo}\label{pro:newConv}
	Suppose that $\phi_\rho(0)>0$ and $\lambda^\ast<\infty$. 
	Consider any $\lambda_1\in[0,\lambda^\ast]$ and let $\lambda_n=\phi_\rho(\lambda_{n-1})$ for $n\geq 2$. Then $(\lambda_n)$ is nondecreasing and converges  to $\lambda^\ast$.
%
\end{propo}
\begin{proof}
By Lemma \ref{pro:lambdastar} (i), the assumption that $\lambda_1\in[0,\lambda^\ast]$ implies that $\lambda_1\leq \phi_\rho(\lambda_1)=\lambda_2$. Since $\phi_\rho$ is a nondecreasing function, by Lemma \ref{pro:lambdastar} (ii) we obtain that $\lambda_2\leq \lambda^\ast$. Arguing inductively, we can show that  $(\lambda_n)$ is a nondecreasing sequence which is bounded by $\lambda^\ast$. 
Thus  $(\lambda_n)$  converges to a limit $\mu\leq \lambda^\ast$. By the definition of $\lambda_n$ and continuity of $\phi_\rho$ we get
	\[
	\mu = \lim_{n\to\infty} \lambda_n = \lim_{n\to\infty} \phi_\rho(\lambda_{n-1}) = \phi_\rho\left(\lim_{n\to\infty} \lambda_{n-1}\right) = \phi_\rho(\mu).
	\]
This and Lemmas \ref{lemma:concavePhi} (c) and \ref{pro:lambdastar} (ii) imply that $\mu = \lambda^\ast$.
%
\end{proof}

\begin{corol}\label{corol:sequenceIEDRVs}

Suppose that $g$ is the $(\rho,H)$-\ldm\  for $(A,B)$. 
Recall $X_n$'s defined in \eqref{y20.2} and suppose that $X_0=0$.

\begin{enumerate}[(i)]
 \item If $g(0)>0$ then for every $n\geq 0$, $X_n$ is an $\IED^{\rho}_H(\lambda_n)$-random variable, where $\lambda_0=0$, $\lambda_1=g(0)$ and
$\lambda_n=\phi_\rho(\lambda_{n-1})$ for $n\geq 2$.
\item If $g(0)>0$ and $\lambda^\ast<\infty$ then the sequence $(\lambda_n)$ in (i) is nondecreasing and converges to  $\lambda^\ast$.
 \item If $\phi_\rho(0)=0$, then $(X_n)$ is a sequence of $\IED^{\rho}_H(0)$-random variables. 
\end{enumerate}

\end{corol}
\begin{proof}
(i) 
We have $X_0=0$, $X_1=B_1$, and $X_2=A_2 B_1+B_2$, so, by Remark \ref{j27.1} and Theorem \ref{thm:inverseDependencyPotentialMain} (since $g(0)>0$),
\begin{align*}
\lim_{x\to 0^+}\frac{-\log \P(X_0<x)}{H(x)}&=0,\\
\lim_{x\to 0^+}\frac{-\log \P(X_1<x)}{H(x)}&=g(0),\\
\lim_{x\to 0^+}\frac{-\log \P(X_2<x)}{H(x)}&=\phi_\rho(g(0)).
\end{align*}
Part (i) follows from Theorem \ref{thm:inverseDependencyPotentialMain}, by induction.

(ii) The assumption that $g(0)>0$ implies that $\phi_\rho(0)>0$. 
Hence, we can apply Proposition \ref{pro:newConv} with $\lambda_1=0$ to conclude that $\phi_\rho(0)\leq\lambda^\ast$. We combine this observation with 
Lemma \ref{lemma:concavePhi} (d) to obtain $g(0)\leq g(0^+)=\phi_\rho(0)\leq \lambda^\ast$. Thus, in the notation of part (i), $\lambda_1 = g(0) \in[0,\lambda^\ast]$. Part (ii) now  follows from Proposition \ref{pro:newConv}.

(iii) If $\phi(0)=0$ then $g(0^+)=g(0)=0$ by Lemma \ref{lemma:concavePhi} (d). The claim now follows from Theorem \ref{thm:inverseDependencyPotentialMain}.
\end{proof}

\section{Solutions to the fixed point equation are IED}\label{sec:perpetuity}

This section is devoted to the proof of the following result.

\begin{thm}\label{thm:solutionOfSFPE}
Suppose that  the $(\rho, H)$-\ldm\ for $(A,B)$ exists.
If $X$ is a solution to \eqref{y20.1}
then $X$ is an $\IED^\rho_H(\lambda^\ast)$ random variable.
\end{thm}

The proof of the theorem will  consist of several lemmas.
All lemmas in this section are based on the same assumptions as those in Theorem \ref{thm:solutionOfSFPE}. The following result was motivated by \cite[Lemma 3]{Grey}. A similar idea was applied  in \cite[Lemma 6.2]{Gamma}.

\begin{lemma}\label{lem:Z}
Suppose that  $c> 0$  satisfies 
	\begin{align}\label{eq:const}
	\phi_\rho(c)>c.
	\end{align}
There exists a nonnegative random variable $Z_c$, independent of
	$(A, B)$, such that
\begin{align}\label{m9.1}
-\log\P(Z_c<\eps)\sim c H(\eps), \quad \eps \downarrow 0,
\end{align}
	and
	\begin{align}\label{eq:Zineq}
	AZ_c+B\geq_{st}Z_c.
	\end{align}
\end{lemma}
\begin{proof}
By Remark \ref{remark:monotoneEquivalents} we may assume without loss of generality  that $H$ is continuous, monotone and $\lim_{t\to\infty}H(t)=0$. Let $Z_0$ be a random variable independent from $(A,B)$, with the distribution defined by $\P(Z_0<\eps)=e^{-c H(\eps)}$ for $\eps>0$. 
By Theorem \ref{thm:inverseDependencyPotentialMain},
\begin{align}\label{m10.1}
\lim_{\eps\to0^+} \frac{ -\log\P(AZ_0+B<\eps)}{H(\eps)} = \phi_\rho(c),
\end{align}
	that is,
	$\P(AZ_0+B<\eps) = \exp\left(-H(\eps)(\phi_\rho(c)+o(1))\right)$.
This, \eqref{eq:const} and \eqref{m10.1} imply that there exists $\eps_0>0$ such that
	\[
	\P(A Z_0+B<\eps) \leq \P(Z_0<\eps),\qquad\forall\,\eps\in(0,\eps_0).
	\]
Let $Z_c$ be a random variable independent from $(A,B)$, with the distribution defined by $\P(Z_c\in \cdot\,) 
	= \P(Z_0\in\cdot\mid Z_0<\eps_0)$. Then we have for $\eps\in(0,\eps_0)$,
	\begin{align*}
	\P(AZ_c+B<\eps) &= \P(AZ_0+B< \eps \mid Z_0<\eps_0) 
	\leq \frac{ \P(AZ_0+B< \eps)}{\P(Z_0<\eps_0)} \\
	&\leq \frac{ \P(Z_0<\eps)}{\P(Z_0<\eps_0)} = \P(Z_c<\eps).
	\end{align*}
	For $\eps\geq \eps_0$ the above inequality holds trivially, since then $\P(Z_c<\eps)=1$. Thus, \eqref{eq:Zineq} is satisfied.
	Finally, note that $\log\P(Z_c<\eps)\sim \log\P(Z_0<\eps)=-c H(\eps)$ as $\eps\to0^+$. This proves \eqref{m9.1}.
\end{proof}

\begin{lemma}\label{lem:liminf}
We have 
$$\liminf_{\eps\to0^+} \frac{-\log \P(X<\varepsilon)}{H(\varepsilon)}\geq \lambda^\ast.$$ 
\end{lemma}
\begin{proof}
If we apply Lemma \ref{pro:p1} with $X_0$ equal to $Z_c$ from Lemma \ref{lem:Z} then we obtain $X \geq _{st}X_0 = Z_c$, for any $c>0$ such that $c<\phi_\rho(c)$. Then, by  Theorem \ref{thm:inverseDependencyPotentialMain}, 
\begin{align}\label{m10.3}
\liminf_{\eps\to 0^+} \frac{-\log\P(X<\eps)}{H(\eps)} \geq \liminf_{\eps\to 0^+} \frac{-\log\P(Z_c<\eps)}{H(\eps)} = c.
\end{align}
By \eqref{m10.2} and Proposition \ref{pro:newConv} (ii),
$\sup\{c\colon \phi_\rho(c)>c\} =\lambda^\ast$. This observation and \eqref{m10.3} imply the lemma.
\end{proof}

\begin{lemma}\label{lem:limsup1}
\begin{align*}
\text{If} \qquad 
s:=\limsup_{\varepsilon\to 0^+} \frac{-\log \P(X<\varepsilon)}{H(\varepsilon)} < \infty
\qquad \text{then} \qquad
\limsup_{\varepsilon\to 0^+} \frac{-\log \P(X<\varepsilon)}{H(\varepsilon)} \leq \lambda^\ast.
\end{align*}
\end{lemma}
\begin{proof}
We have 
\begin{align}\label{m10.4}
\P(X<\varepsilon)=\P(AX+B<\varepsilon)\geq \P(X<\varepsilon y)\P(\varepsilon Ay+B<\varepsilon),
\end{align}
and this gives us
$$s\leq \limsup_{\eps\to0^+}\frac{-\log \P(X<\varepsilon y)}{H(\varepsilon y)}\cdot \frac{H(\varepsilon y)}{H(\varepsilon)}+\limsup_{\eps\to0^+}\frac{-\log \P(\varepsilon Ay+B<\varepsilon)}{H(\varepsilon)}\leq \frac{s}{y^\rho}+g(y).$$ 
Hence, for $y>1$ we have $s\leq g(y) y^\rho/(y^\rho-1)$ and thus $s\leq \lambda^\ast$.
\end{proof}

\begin{lemma}\label{lem:limsup2}
Assume that $\lambda^\ast<\infty$. Then 
$$
\limsup_{\eps\to0^+}\frac{-\log\P(X<\eps)}{H(\eps)}<\infty.
$$
\end{lemma}
\begin{proof}
Since $\lambda^\ast<\infty$, there exists $y>1$ such that $g(y)<\infty$. Then, for any $\eta>0$, there exists $\eps_0$ such that for all $\eps\leq\eps_0$,
	\begin{align}\label{eq:ineq3}
	\frac{-\log\P(\eps A y+B<\eps)}{H(\eps)}\leq g(y)+\eta.
	\end{align}
It follows from \eqref{m10.4} that
	$$
	-\log\P(X<\eps)+\log\P(X<\eps y)\leq -\log\P(\eps A y+B<\eps).
	$$
Substituting $\eps y^k$ for $\eps$ in the last formula yields
	$$
	-\log\P(X<\eps y^k)+\log\P(X<\eps y^{k+1})\leq -\log\P(\eps y^k A y+B<\eps y^k).
	$$
If we further assume that $\eps y^k\leq\eps_0$, by \eqref{eq:ineq3}, we arrive at
	$$
	-\log\P(X<\eps y^k)+\log\P(X<\eps y^{k+1})\leq (g(y)+\eta) H(\eps y^k).
	$$
	The telescoping sum argument gives 
	$$
	-\log\P(X<\eps)+\log\P(X<\eps y^{n+1})\leq (g(y)+\eta)\sum_{k=0}^{n} H(\eps y^k),
	$$
	provided $\eps y^n\leq \eps_0$.
This condition is satisfied if we set $n=n_\eps=\left\lfloor \log(\eps_0/\eps)/\log(y)\right\rfloor$. With this choice of $n$ we also have
	$\eps y^{n_\eps+1}\geq \eps_0$. Thus, we obtain
	$$
\limsup_{\eps\to0^+}	\frac{-\log\P(X<\eps)}{H(\eps)} \leq (g(y)+\eta) \limsup_{\eps\to0^+} \sum_{k=0}^{n_\eps} \frac{H(\eps y^k)}{H(\eps)}.
	$$
	Finally, by Potter bounds (e.g. \cite[Theorem 1.5.6]{regularVariation}) we have $H(\eps y^k)/H(\eps)\leq C y^{-k(\rho-\delta)}$ for any $C>1$, $\delta\in(0,\rho)$ and $\eps$ small enough. This ensures convergence of the series on the right hand side above.	
\end{proof}

\begin{proof}[Proof of Theorem \ref{thm:solutionOfSFPE}]
 Theorem \ref{thm:solutionOfSFPE} follows from Lemmas \ref{lem:liminf}, \ref{lem:limsup1} and \ref{lem:limsup2} in the case when $\lambda^\ast<\infty$.
If $\lambda^\ast=\infty$, then the assertion of Lemma \ref{lem:limsup1} is  satisfied. Therefore, Theorem \ref{thm:solutionOfSFPE} holds in either case.
\end{proof}

\section{Positive quadrant dependent coefficients}\label{sec:examples}

We will now illustrate the concepts of \ldm \ $g$ and its transform $\phi_\rho$
by applying them to certain classes of vectors $(A,B)$. 
In this section we will find a formula for \ldm \  $g$
in the case when the $A$ and $B$ are positively quadrant dependent and 
$B$ is an $\IED_H^\rho(\lambda)$ random variable. The equation \eqref{y20.1} with  coefficients satisfying these assumptions was studied in \cite{IEDPaper} using different methods.
We will  show how the results in \cite{IEDPaper} relate to the \ldm \  $g$ and its transform $\phi_\rho$.

\begin{defin}
We  call random variables $A$ and $B$ positively quadrant dependent if
\begin{equation}
 \P(A>a,B>b)\geq \P(A>a)\P(B>b),\label{eq:positivelyQuadrantDependent}
\end{equation}
for all $a,b\in \R$. 
\end{defin}

If two random variables are independent then they are also positively quadrant dependent. For the proof of the following lemma, see \cite[Lemma 7.3]{IEDPaper}.

\begin{lemma}\label{n11.1}
Random variables $A$ and $B$ are positively quadrant dependent if and only if
\begin{equation}
 \P(A\leq a,B\leq b)\geq \P(A\leq a)\P(B \leq b)\label{eq:positivelyQuadrantDependentCharacterization}
\end{equation}
for all $a,b\in \R$. 
\end{lemma}

\begin{propo}\label{y31.2}
Suppose that $B$ is an $\IGFT^\rho_H(\gamma)$-random variable, $(A,B)$ are positively quadrant dependent, and let
$$a= \essinf(A)= \sup\{x\in \R : \P(A< x)=0\}.$$
Then 
$$g(y) = \left\{\begin{array}{cc}
          \gamma(1-ay)^{-\rho},& y\in[0,1/a);\\
	  \infty,& y\geq 1/a.
         \end{array}\right.
$$  
\end{propo}
\begin{proof}
Since $A$ is nonnegative, $a\geq 0$. The definition of $a$ implies that 
$$\P(\varepsilon Ay+B<\varepsilon)\leq \P(\varepsilon ay+B<\varepsilon).$$
This, the assumption that  $B$ is an $\IGFT^\rho_H(\gamma)$-random variable, and  Definition \ref{definition:IEDRandomVariable} show that, for $y\in [0,1/a)$,
\begin{align*}
g(y)&=\lim_{\varepsilon\to 0^+} \frac{-\log\P(\varepsilon Ay+B<\varepsilon)}{H(\eps)}
\geq \lim_{\varepsilon\to 0^+} \frac{-\log\P(\varepsilon ay+B<\varepsilon)}{H(\eps)}\\
&= \lim_{\varepsilon\to 0^+} \frac{-\log\P(B<\varepsilon(1-ay))}{H(\eps)}
 = \lim_{\varepsilon\to 0^+} \frac{-\log\P(B<\varepsilon(1-ay))}{H(\eps(1-ay))} \frac{H(\eps(1-ay))}{H(\eps)}\\
&= \gamma (1-ay)^{-\rho}.
\end{align*}
With the convention 
that $\log 0=-\infty$, we get for $y\geq 1/a$,
\begin{align*}
g(y)&=\lim_{\varepsilon\to 0^+} \frac{ -\log\P(\varepsilon Ay+B<\varepsilon)}{H(\eps)}
\geq \lim_{\varepsilon\to 0^+} \frac{ -\log\P(\varepsilon ay+B<\varepsilon)}{H(\eps)}\\
&\geq \lim_{\varepsilon\to 0^+} \frac{-\log\P(B<0)}{H(\eps)}= \infty.
\end{align*}

To obtain the upper bound, consider any $y<1/a$ and find $\delta_0>0$ such that for
$\delta \in (0,\delta_0)$ we have $y<1/(a+\delta)$.
Then for  $\delta\in (0,\delta_0)$,
\begin{align*}
 \P(\varepsilon Ay+B\leq \varepsilon) &\geq \P(\varepsilon Ay+B\leq \varepsilon, A\in [a,a+\delta])\\
&\geq  \P(\varepsilon (a+\delta)y+B\leq \varepsilon, A\in [a,a+\delta])\\
&\geq \P(\varepsilon (a+\delta)y+B\leq \varepsilon)\P( A\in [a,a+\delta]),
\end{align*}
where the last inequality follows from Lemma \ref{n11.1}.
By definition of $a$, we have $\P( A\in [a,a+\delta))>0$, so
\begin{align*}
g(y)&=\lim_{\varepsilon\to 0^+} \frac{-\log\P(\varepsilon Ay+B<\varepsilon)}{H(\eps)}\\
&\leq \lim_{\varepsilon\to 0^+} \frac{-\log\big(\P(\varepsilon (a+\delta)y+B\leq \varepsilon)\P( A\in [a,a+\delta])
\big)}{H(\eps)}\\
&=\lim_{\varepsilon\to 0^+} \frac{-\log\P(\varepsilon (a+\delta)y+B\leq \varepsilon)}{H(\eps)}\\
& = \lim_{\varepsilon\to 0^+} \frac{-\log\P(B\leq \varepsilon(1- (a+\delta)y))}{H(\eps (1- (a+\delta)y))} \frac{H(\eps (1- (a+\delta)y))}{H(\eps)}\\
&= \gamma (1-(a+\delta)y))^{-\rho}.
\end{align*}
Letting $\delta \to 0^+$, we obtain $g(y)\leq \gamma(1-ay)^{-\rho}$, for $y\in [0,1/a)$.
\end{proof}

\begin{propo}\label{propo:phiABPQD}
Under assumptions of Proposition \ref{y31.2}, 
$$\phi_\rho(\lambda) =\left(\gamma^{\frac{1}{1+\rho}}+a^{\frac{\rho}{1+\rho}}\lambda^{\frac{1}{1+\rho}}\right)^{1+\rho},$$
and 
\begin{equation}
 \lambda^\ast=\left\{\begin{array}{cl}
                   \gamma\left(1-a^{\frac{\rho}{1+\rho}}\right)^{-(1+\rho)}, &\textrm{for}\ a<1;\\
		   \infty,& \textrm{for}\ a\geq 1.
                  \end{array}
\right.\label{eq:lambda*}
\end{equation}

\end{propo}
\begin{proof}
Since $g(y)$ takes finite values only on the interval $[0,1/a)$, we need to find the minimum of the function
$$y\mapsto g(y)+\frac{\lambda}{y^\rho}=\frac{\gamma}{(1-ay)^\rho}+\frac{\lambda}{y^\rho}$$
on the interval $(0,1/a)$. One can show that that minimum is attained at 
$$y_1=\frac{\lambda^{\frac{1}{1+\rho}}}{(\gamma a)^{\frac{1}{1+\rho}}+a\lambda^{\frac{1}{1+\rho}}}=\frac{1}{a}\cdot \frac{a\lambda^{\frac{1}{1+\rho}}}{(\gamma a)^{\frac{1}{1+\rho}}+a\lambda^{\frac{1}{1+\rho}}}\in (0,1/a).$$
Straightforward calculations yield the formulas for $\phi_\rho(\lambda) = \frac{\gamma}{(1-ay_1)^\rho}+\frac{\lambda}{y_1^\rho}$ and $\lambda^\ast = \phi_\rho(\lambda^\ast)$ given in the proposition.
\end{proof}

We will illustrate the meaning of $\lambda^\ast$
by two results borrowed from \cite{IEDPaper}; they were stated in that paper as Theorems 7.6 and 7.8. The versions given below include $\lambda^\ast$, the parameter introduced only in this paper. The versions given in \cite{IEDPaper} and these in the present paper are equivalent due to \eqref{eq:lambda*}.

\begin{thm}\label{thm:SolutionToSFE}
Assume that 
\begin{enumerate}[(i)]
 \item $A$ and $B$ are nonnegative and positively quadrant dependent.
 \item $\E[\log A]<0$ and $\E[\log^+B]<\infty$.
 \item $B$ is an \IGFT$^\rho_H(\gamma)$-random variable.
\end{enumerate}

Then

(a) The random variable $S$ defined in \eqref{eq:defS} is \IGFT$^\rho_H(\lambda^\ast)$.

(b) The  equation \eqref{y20.1} 
has a unique solution with the same distribution as that of $S$.
\end{thm}

\begin{thm}
 \label{thm:lowerEnvelope3}
Suppose that
\begin{enumerate}[(i)]
 \item $A$ and $B$ are nonnegative and positively quadrant dependent random variables. 
 \item There exists $\beta\in (0,1)$ such that $A\leq \beta$, a.s. 
\item $\E[(\log^+B)^s]<\infty$ for all $s>0$.
 \item $B$ is an \IGFT$^\rho_H(\gamma)$-random variable.
\end{enumerate}
If the sequence $(X_n)$ is defined as in \eqref{y20.2} then
$$
\liminf_{n\to\infty} \frac{X_n}{H^{-1}(\log n)} =(\lambda^\ast)^{1/\rho},\ a.s.
$$
\end{thm}

The last two theorems were proved in \cite{IEDPaper} using  techniques tailored for the assumption that $A$ and $B$
were positive quadrant dependent. Part (b) of Theorem \ref{thm:SolutionToSFE}. is special case of Theorem \ref{thm:solutionOfSFPE}.
In the next section, we will prove Theorem \ref{thm:LowerEnvelope}, which is a much more general version of Theorem \ref{thm:lowerEnvelope3}.

\section{Local dependence measure and logarithmic lower envelope}\label{sec:logarithmicLowerEnvelope}

Recall the sequence $(X_n)$ defined in \eqref{y20.2} and set $X_0=0$.

\begin{thm}\label{thm:LowerEnvelope}
Assume that $\E[\log A] <0$ and $\E[\log^+ B]<\infty$. Suppose that $g$ is the $(\rho, H)$-\ldm\  for  $(A,B)$, $g(0)>0$ and $\lambda^\ast\in(0,\infty)$.
Then 
$$\liminf_{n\to \infty}\frac{X_n}{H^{-1}(\log n)}= (\lambda^\ast)^{1/\rho}.$$
\end{thm}

The proof of the theorem will  consist of several lemmas.
All lemmas in this section  implicitly make the same assumptions as those in Theorem \ref{thm:LowerEnvelope}.

\begin{lemma}
\label{lemma:easySideOfBorelCantelli2}
(i) For every $\varepsilon>0$,
$$\left\{X_n\leq H^{-1}\left(\frac{(1+\varepsilon)\log n}{\lambda^\ast}\right)\right\}$$
happens finitely often almost surely.

(ii)  We have
$$\liminf_{n\to\infty}\frac{X_n}{H^{-1}(\log n)}  \geq (\lambda^\ast)^{1/\rho}\ \text{a.s.}$$
\end{lemma}
\begin{proof}
(i) For any $\eps>0$
there exists $\delta\in(0,1)$ such that $\gamma := (1-\delta)(1+\varepsilon)>1$. 
Recall the notation from
 Corollary \ref{corol:sequenceIEDRVs}. 
The corollary  shows that $\lambda_n \uparrow \lambda^\ast$. Hence
there exist $C_\delta$, $n_0$ and $x_0>0$ such that $\P(X_{n_0}\leq x)\leq C_\delta e^{-\lambda^\ast(1-\delta) H(x)}$ for all $x\in (0,x_0)$.
By Lemma \ref{pro:p1}, for $n\geq n_0$ and $x\in (0,x_0)$,
$$\P(X_n\leq x)\leq C_\delta e^{-\lambda^\ast(1-\delta) H(x)}.$$
It follows that, for large $n$,
\begin{align*}
 \P\left(X_n\leq H^{-1}\left(\frac{(1+\varepsilon)\log n}{\lambda^\ast}\right)\right)
\leq C_\delta e^{-(1-\delta)(1+\varepsilon)\log n}=C_{\delta}n^{-\gamma}.
\end{align*}
Hence, 
$$\sum_{n=1}^{\infty}\P\left(X_n\leq H^{-1}\left(\frac{(1+\varepsilon)\log n}{\lambda^\ast}\right)\right)<\infty,$$
and the claim follows by the Borel-Cantelli lemma.

(ii) Part (i) implies that for every $\eps >0$, a.s.,
$$\liminf_{n\to\infty}\frac{X_n}{H^{-1}\left((1+\varepsilon)(\log n)/\lambda^\ast\right)}  \geq 1.$$
But $H^{-1}$ is regularly varying with index $-1/\rho$ at infinity and thus
\begin{align*}
H^{-1}\left(\frac{(1+\varepsilon)\log n}{\lambda^\ast}\right) \sim \left(\frac{\lambda^\ast}{1+\varepsilon}\right)^{1/\rho} H^{-1}\left(\log n\right).
\end{align*}
Hence, a.s.,
$$\liminf_{n\to\infty}\frac{X_n}{H^{-1}\left(\log n\right)}  \geq  \left(\frac{\lambda^\ast}{1+\varepsilon}\right)^{1/\rho}.$$
Part (ii) follows by letting $\eps\to0$. 
\end{proof}

\begin{lemma}\label{lemma:LowerBoundOnConditionalProbabilities}
For all $n\geq 1$, $y>0$ and $\eps>0$  we have, a.s.,
\begin{align}\label{j15.2}
\P(X_n<\eps \mid X_0)\geq \1_{\left[0,\eps y^n\right)}(X_0)\prod_{k=0}^{n-1} \P(\eps y^{k}Ay+B< \eps y^k).
\end{align}
\end{lemma}

\begin{proof}
	We have
	\begin{align*}
	\P&(X_n< \eps  \mid X_0)\geq \P(A_n X_{n-1}+B_n< \eps, X_{n-1}<\eps y  \mid X_0)\\
&	\geq \P(\eps A_n y+B_n< \eps, X_{n-1}<\eps y  \mid X_0)
	=   \P(\eps A_n y+B_n< \eps) \P(X_{n-1}< \eps y \mid X_0) \\
&	=   \P(\eps A y+B< \eps) \P(X_{n-1}< \eps y \mid X_0) .
	\end{align*}
	The assertion follows by induction.
\end{proof}

We state, without formal proofs, three simple results, for reference.
Recall that $\lambda^\ast = \inf_{y>1}\left\{\frac{g(y) y^\rho}{y^\rho-1}\right\}$.
\begin{lemma}\label{cor:lambda}
Assume that $\lambda^\ast\in(0,\infty)$. For any $\delta>0$, there exists $y_\ast>1$ such that
\[
\lambda^\ast\leq \frac{g(y_\ast) y_\ast^\rho}{y_\ast^\rho-1}\leq \lambda^\ast(1+\delta).
\]
\end{lemma}
\begin{lemma}\label{cor:g}
For any $\delta>0$ and $y>0$, there exists $\eps_0>0$ such that
for all $\eps\in(0,\eps_0)$,
\[
\P(\eps A y+B<\eps)\geq e^{-(1+\delta) g(y) H(\eps) }.
\]
\end{lemma}
Recall that $H(\eps y)\sim y^{-\rho}H(\eps)$ as $\eps\to0^+$. The following result is an application of Potter bounds to function $H$ (see \cite[Theorem 1.5.6]{regularVariation}).
\begin{lemma}\label{cor:H}
For any $\delta>0$, $y>1$ and $\eta\in(0,\rho)$, there exists $\eps_1$ such that
\[
\frac{H(\eps y)}{H(\eps )} \leq (1+\delta) y^{-\rho+\eta}
\]	
for all $\eps\in(0,\eps_1/y)$.
\end{lemma}

\begin{lemma}\label{lemma:prepareToBC}
	For any $\delta>0$ and $n \geq 1 $, there exist $y_\ast>1$ and $\tilde{\eps}>0$ such that
	\begin{align*}
	\P(X_n<\eps \mid X_0)\geq \1_{\left[0,\eps y_\ast^n\right)}(X_0) 
	\exp(-(1+\delta) \lambda^\ast H(\eps)),
	\end{align*}
	provided $\eps y_\ast^{n-1}< \tilde{\eps}$. 
\end{lemma}
\begin{proof}
Fix $\alpha >0$ and let $y_\ast>1$ be as in Lemma \ref{cor:lambda}. By Lemma \ref{cor:g} there exists $\eps_0>0$ such that
\[
\P(\eps A y_\ast+B<\eps)\geq 
\exp(-(1+\alpha) g(y_\ast)  H(\eps) )
\]	
for all $\eps\in(0,\eps_0)$. Thus, by Lemma \ref{lemma:LowerBoundOnConditionalProbabilities}, we obtain
\begin{align*}
\P(X_n<\eps \mid X_0)\geq \1_{\left[0,\eps y_\ast^n\right)}(X_0) 
\exp\left(-(1+\alpha) g(y_\ast) \sum_{l=0}^{n-1}   H(\eps y_\ast^l) 
\right),
\end{align*}
provided $\eps y_\ast^{n-1}< \eps_0$. By Lemma \ref{cor:H}, for $\eta\in(0,\rho)$,
\[
H(\eps y_\ast^k)\leq (1+\alpha) y_\ast^{-k(\rho-\eta)} H(\varepsilon),\qquad k=0,1,\ldots,n-1,
\]
as long as $\eps y_\ast^{n-1}<\eps_1$. Hence, if $\eps y_\ast^{n-1}<\tilde{\eps} :=\min\{\eps_0,\eps_1\}$, then
\begin{align}\label{m10.5}
\P(X_n<\eps \mid X_0)\geq \1_{\left[0,\eps y_\ast^n\right)}(X_0) 
\exp\left(-(1+\alpha)^2 g(y_\ast) \sum_{k=0}^{n-1}  y_\ast^{-k(\rho-\eta)}  H(\eps) \right).
\end{align}
By Lemma \ref{cor:lambda}, for sufficiently small $\eta >0$,
\begin{align*}
g(y_\ast) \sum_{k=0}^{n-1}  y_\ast^{-k(\rho-\eta)} &= g(y_\ast) \frac{y_\ast^{\rho-\eta}}{y_\ast^{\rho-\eta}-1}\left(1-y_\ast^{-n(\rho-\eta)}\right) \leq (1+\alpha)  g(y_\ast) \frac{y_\ast^{\rho}}{y_\ast^{\rho}-1} \\
&\leq (1+\alpha)^2 \lambda^\ast.
\end{align*}
This and \eqref{m10.5} show that
\begin{align*}
\P(X_n<\eps \mid X_0)\geq \1_{\left[0,\eps y_\ast^n\right)}(X_0) 
\exp\left(-(1+\alpha)^4 \lambda^\ast H(\eps) \right).
\end{align*}
The lemma follows if we take $(1+\alpha)^4 = 1+\delta$.
\end{proof}

We will need the following version of the Borel-Cantelli Lemma.
\begin{lemma}\label{BorelCantelliLemmaConditional}\begin{enumerate}[(a)]
		\item Suppose that $(\F_n)$ is a filtration such that $\F_0=\{\emptyset,\Omega\}$, and  $A_n\in \F_{n}$ for $n\geq 0$. Then 
		$$\{A_n\ \ i.o.\}=\left\{\sum_{n=1}^{\infty}\P(A_n\mid \F_{n-1})=\infty\right\}.$$
		\item Suppose that  $(X_n)$ is a Markov process with respect to a filtration $(\F_n)$ such that $\F_0=\{\emptyset,\Omega\}$, and $A_n\in\sigma(X_n)$ for $n\geq 1$.
		Then 
		$$\{A_n\ \ i.o.\}=\left\{\sum_{n=1}^{\infty}\P(A_n\mid X_{n-1})=\infty\right\}.$$
	\end{enumerate}
\end{lemma}

\begin{proof}
For (a), see \cite[Thm. 5.1.2]{Chandra}. Part (b) is an easy corollary of  (a).
\end{proof}

We state the following  well-known Kronecker's lemma without proof.
\begin{lemma}\label{j15.3}
 If $a_n\uparrow \infty$ and 
$\sum_{n=1}^{\infty} x_n/a_n$
converges then 
$\lim_{n\to\infty}\frac{1}{a_n}\sum_{m=1}^{n} x_m= 0$.
\end{lemma}

We will need the following result on the ergodicity for subsequences of the iterated stochastic sequence.

\begin{lemma}\label{prop:subsequncesErgodicity}
Suppose that $X$ is a solution to \eqref{stochasticFixedPointEquation}.
 For any bounded uniformly continuous functions $f$ on $\R$
and any increasing integer sequence $(n_k)$, a.s.,
\begin{equation}\label{eq:ergodicSubsequence}
 L(f):=\limsup_{m\to \infty}\frac{1}{m}\sum_{k=1}^{m}f(X_{n_k})\geq \E[f(X)]\geq l(f):=\liminf_{m\to \infty}\frac{1}{m}\sum_{k=1}^{m}f(X_{n_k}).
\end{equation}
Moreover, $L(f)$ and $l(f)$ are constants a.s.
\end{lemma}
\begin{proof}
 For $r \geq 1 $, we define 
$$X_n^r:=\left\{\begin{array}{cl}
                 0,& n\leq r;\\
		 A_nX^r_{n-1}+B_n& n>r.
                \end{array}
\right.$$
We have assumed that $\E[\log A] <0$ so $\lim_{n\to\infty}\prod_{j=r+1}^{n}A_j = 0$, a.s. Therefore, when $n\to \infty$, a.s.,
$$X_{n}-X^{r}_n= \left(\prod_{j=r+1}^{n}A_j\right) X_{r} \to 0.$$
Hence $\lim_{n\to \infty} f(X_n)-f(X_n^{r})=0$, a.s., and it follows that, a.s.,
$$\lim_{m\to \infty}\frac{1}{m}\sum_{k=1}^{m}f(X_{n_k})-f(X_{n_k}^r)=0.$$
This implies that, a.s.,
\begin{align}\label{m5.1}
\limsup_{m\to \infty}\frac{1}{m}\sum_{k=1}^{m}f(X_{n_k})
&=\limsup_{m\to \infty}\frac{1}{m}\sum_{k=1}^{m} f(X_{n_k}^r),\\
\liminf_{m\to \infty}\frac{1}{m}\sum_{k=1}^{m}f(X_{n_k})
&=\liminf_{m\to \infty}\frac{1}{m}\sum_{k=1}^{m} f(X_{n_k}^r).
\label{m5.2}
\end{align}
For every fixed $r>0$,
the random variables on the right hand sides of \eqref{m5.1} and \eqref{m5.2} are measurable with respect to the $\sigma$-field $\G_r:= \sigma((A_n,B_n): n\geq r)$. Thus the same applies to the random variables on the left hand sides of \eqref{m5.1} and \eqref{m5.2}.
Hence, these random variables are measurable with respect to the $\sigma$-field $\G_{\infty}:=\bigcap_{r=1}^{\infty} \F_r$. By the Kolomogorov 0-1 law, random variables on both sides of \eqref{m5.1} and \eqref{m5.2} are constant, a.s. 

By Corollary \ref{thm:uniquenessAndConvergence} (i), $X_n \to X$ in distribution. This implies that $\lim_{n\to \infty}\E[f(X_n)]=\E[f(X)]$. We combine this observation with Fatou's Lemma ($f$ need not be nonnegative, but it is bounded) to obtain,
\begin{align*}
\liminf_{m\to \infty}\frac{1}{m}\sum_{k=1}^{m} f(X_{n_k})
&=
\E\left[
\liminf_{m\to \infty}\frac{1}{m}\sum_{k=1}^{m} f(X_{n_k})
\right] \leq 
\lim_{m\to\infty}
\E\left[
\frac{1}{m}\sum_{k=1}^{m}f(X_{n_k})\right] \\
&= \E[f(X)].
\end{align*}
This proves the inequality on the right hand side of \eqref{eq:ergodicSubsequence}. The inequality on the left hand side follows by applying the claim to $-f$ in place of $f$.
\end{proof}

\begin{lemma}
\label{lemma:hardSideOfBorelCantelli}\ 

(i) For every $\eps>0$,
$$\left\{X_n\leq H^{-1}\left(\frac{\log n}{\lambda^\ast(1+\eps)}\right)\right\}$$
happens infinitely often almost surely.

(ii)  Almost surely,
$$\liminf_{n\to\infty}\frac{X_n}{H^{-1}(\log n)}  \leq (\lambda^\ast)^{1/\rho}.$$
\end{lemma}

\begin{proof}
Fix any $\eps >0$.
Let $(k_n)_n$ be a strictly increasing sequence of integers. Since $(X_{k_{n+1}-k_n} \mid X_0)\stackrel{d}{=}(X_{k_{n+1}} \mid X_{k_n})$ for any $\delta>0$ and $n\geq 1$, by Lemma \ref{lemma:prepareToBC} there exist $y_*>1$ and $\tilde{\eps} >0$ such that,  a.s., for $t>0$,
\begin{align}\label{eq:firstI}
\P(X_{k_{n+1}}<t \mid X_{k_n})\geq \1_{\left[0,t\, y_\ast^{k_{n+1}-k_n}\right)}(X_{k_n})  e^{-(1+\delta) \lambda^\ast H(t)}
\end{align}
provided 
\[
t\, y_\ast^{k_{n+1}-k_n-1}<\tilde{\eps}.
\] 
By Lemma \ref{propo:main:kn} we can choose the sequence $(k_n)$, so it satisfies for each $n \geq 1 $,
\begin{align*}
\begin{split}
H^{-1}\left( \frac{\log k_{n+1}}{\lambda^\ast(1+\eps)} \right)y_\ast^{k_{n+1}-k_n-1}<\tilde{\eps}, \\
H^{-1}\left( \frac{\log k_{n+1}}{\lambda^\ast(1+\eps)} \right)y_\ast^{k_{n+1}-k_n} \geq c,
\end{split}
\end{align*}
where $c\in(0,\tilde{\eps} y_\ast)$.
Then, taking $t = H^{-1}\left( \frac{\log k_{n+1}}{\lambda^\ast(1+\eps)} \right)$ in \eqref{eq:firstI}, we have, a.s.,
\begin{align}\label{eq:SecondI}
\P\left(X_{k_{n+1}}<H^{-1}\left( \frac{\log k_{n+1}}{\lambda^\ast(1+\eps)} \right) \mid X_{k_n}\right)\geq \1_{[0,c)}(X_{k_n}) \frac{1}{k_n^\gamma},
\end{align}
where $\gamma=\frac{1+\delta}{1+\eps}$. Take $\delta <\eps$ so that $\gamma<1$.
By Lemma \ref{propo:main:kn}, there  exists $K>0$ such that $k_n^{\gamma}\leq K(n+1)$ for all $n$.

We have, a.s.,
\begin{align}\notag
 \limsup_{m\to\infty}\frac{1}{k_m^\gamma}\sum_{n=0}^{m}\1_{[0,c)}(X_{k_n})
&\geq
\limsup_{m\to\infty}\frac{K^{-1}}{m+1}\sum_{n=0}^{m}\1_{[0,c)}(X_{k_n})
\geq
\limsup_{m\to\infty}\frac{K^{-1}}{m+1}\sum_{n=0}^{m}f_c(X_{k_n}) \\
\label{eq:applyingSubsequenceErgodicity}
&\geq \E[f_c(X)]/K\geq  \P(X <c/2)/K>0,
\end{align}
where the first inequality on the second line of \eqref{eq:applyingSubsequenceErgodicity} follows from Lemma \ref{prop:subsequncesErgodicity} applied to  the function
\begin{align*}
f_c(x) =
\begin{cases}
1, &x<c/2;\\
		2(c-x)/c, &x\in [c/2,c];\\
		0, &x>c.
\end{cases}
\end{align*}
The last inequality in \eqref{eq:applyingSubsequenceErgodicity} follows from Theorem \ref{thm:solutionOfSFPE} because we assumed that $\lambda^\ast\in(0,\infty)$ in  Theorem \ref{thm:LowerEnvelope}.

Kronecker's lemma (Lemma \ref{j15.3}) and \eqref{eq:applyingSubsequenceErgodicity} imply that 
$$\sum_{n=0}^{\infty}\1_{[0,c)}(X_{k_n}) \frac{1}{k_n^\gamma}=\infty, \ \text{a.s.}$$
Hence, in view of \eqref{eq:SecondI}, a.s.,
$$\sum_{n=1}^{\infty}\P\left(X_{k_{n+1}}<H^{-1}\left( \frac{\log k_{n+1}}{\lambda^\ast(1+\eps)} \right) \mid X_{k_n}\right)= \infty.$$
This and Lemma \ref{BorelCantelliLemmaConditional} (b) imply part (i) of the present lemma.

Recall that $H^{-1}$ is a regularly varying function at $\infty$ with index $-1/\rho$ to see that
part  (ii) of the lemma follows from part (i).
\end{proof}
\begin{proof}[Proof of Theorem \ref{thm:LowerEnvelope}]
The theorem follows from Lemmas \ref{lemma:easySideOfBorelCantelli2} and  \ref{lemma:hardSideOfBorelCantelli}  
\end{proof}

\section{Application to Fleming-Viot type process}\label{sec:FlemingViot}

This section is devoted to the proof of Theorem \ref{thm:LawOfTheIteratedLogartihm}, a version of the Law of Iterated Logarithm for a Fleming-Viot type process. This result was the primary motivation for introducing and analyzing the ``local dependence measure.''

Fleming-Viot type processes were originally defined in \cite{BHM}. The specific model discussed below is close to those in \cite{extinctionOfFlemingViot}. Under mild assumptions, it was proved in \cite{BB18} that the Fleming-Viot process has a unique spine, i.e., a trajectory inside the branching tree that never hits the boundary of the domain where the process is confined. It was proved in \cite{BB18}, for a Fleming-Viot process on a finite state space, that the distribution of the spine converges to the distribution of the driving process conditioned to never exit the domain, when the number of individuals in the population grows to infinity. We do not know whether a similar result holds for the spine in the specific model discussed below, with the population size fixed and equal to two. The LIL proved in Theorem \ref{thm:LawOfTheIteratedLogartihm} is the first step in our program to analyze this particular spine in detail.

We will now define a Fleming-Viot process and  other elements of the model.
Informally, the process consists of two independent Brownian particles starting at the same point in $(0,\infty)$. At the time when one of them hits 0, it is killed and the other one branches into two particles. The new particles start moving as independent Brownian motions and the scheme is repeated.

On the formal side, let
$(W_1(t):t\geq 0)$ and $(W_2(t):t\geq 0)$ be two independent Brownian motions starting from $W_1(0) =W_2(0)=1$. Let
\begin{align*}
T_0&=0,\\
Y_0&=1,\\
\tau_j &= \inf\{t\geq 0: W_j(t) =0\}, \qquad j=1,2,\\
T_1&=\min(\tau_1,\tau_2),\\
Y_1&=\max(W_1(T_1),W_2(T_1)),
\end{align*}
and for $k\geq 2$,
\begin{align*}
T_{k}&=\inf\{t>T_{k-1} : \min(W_1(t)-W_1(T_{k-1})+Y_{k-1},W_2(t)-W_2(T_{k-1})+Y_{k-1}) =0 \},\\
Y_k&=\max(W_1(T_k)-W_1(T_{k-1})+Y_{k-1},W_2(T_k)-W_2(T_{k-1})+Y_{k-1}).
\end{align*}
It follows from the proof of Theorem 1.4 in \cite{BHM} that $T_k\to \infty$, a.s. Hence, for any $t\geq 0$ we can find $j$ such that $t\in [T_{j-1},T_{j})$. Then we set
\begin{align}\label{j18.1}
\cY(t)=
(Y_1(t),Y_2(t))&=(W_1(t)-W_1(T_{j-1})+Y_{j-1},W_2(t)-W_2(T_{j-1})+Y_{j-1}).
\end{align}
This completes the definition of $\{\cY(t), t\geq 0\}$, an example of a Fleming-Viot process. Let $Z(t) = \max(Y_1(t),Y_2(t))$ be the spine and note that $Z(T_k) = Y_k$ for all $k$. 

The following is the main result of this section.
\begin{thm}\label{thm:LawOfTheIteratedLogartihm}
Almost surely, 
\begin{align}\label{j16.1}
\limsup_{n\to\infty}\frac{Y_n}{\sqrt{2 T_n\log \log T_n}}=1.
\end{align}
\end{thm}

We note that the Law of Iterated Logarithm stated in \eqref{j16.1} indicates (but does not prove) that the spine $Z(t)$ satisfies the same  Law of Iterated Logarithm as the three-dimensional Bessel process, which is known to have the same distribution as the  one-dimensional Brownian motion conditioned not to hit 0. Hence, it is possible that the spine $Z(t)$ is distributed, at least in an asymptotic or approximate sense, as the driving Brownian motion $W_1(t)$ conditioned not to return to 0.
We plan to investigate this question in a forthcoming paper.

The remaining part of this section will be devoted to the proof of Theorem \ref{thm:LawOfTheIteratedLogartihm}, presented as a sequence of lemmas.
The formulas in the first of the lemmas are taken from \cite{KS}, Chapter 2, Remark 8.3 and Problem 8.6.
\begin{lemma}\label{j16.2}
If $W_1(0)=1$ then for $y,t>0$,
\begin{align*}
 &\P(\tau_1 \in dt ) = \frac{1}{\sqrt{2\pi t^3}}e^{-1/2t}dt,\\
 &\P(W_1(t)\in dy, \tau_1>t )
  = \frac 1 {\sqrt{2\pi t}}
 \left( \exp\left(-\frac{(1-y)^2}{2t}\right)-\exp\left(-\frac{(1+y)^2}{2t}\right)\right) dy.
\end{align*}
\end{lemma}

\begin{lemma}\label{lemma:OneStoppedOneKilled}
If $W_1(0)=W_2(0)=1$ then for $y,t>0$,
\begin{align*}
&\P(W_1(\tau_2) \in dy, \tau_2\leq t,  \tau_1>\tau_2)\\
&\qquad = \frac{1}{\pi}\left[\frac{\exp(-((1-y)^2+1)/(2t)}{(1-y)^2+1}-\frac{\exp(-((1+y)^2+1)/(2t)}{(1+y)^2+1}\right]dy.
\end{align*}
\end{lemma}
\begin{proof}
 We use Lemma \ref{j16.2} as follows,
\begin{align*}
 & \P(W_1(\tau_2) \in dy, \tau_2\leq t,  \tau_1>\tau_2)
= \int_0^{t}\P(W_1(s) \in dy, \tau_1>s)\P(\tau_2 \in ds)\\
&\ = \int_0^{t}
\frac 1 {\sqrt{2\pi s}}
 \left( \exp\left(-\frac{(1-y)^2}{2s}\right)-\exp\left(-\frac{(1+y)^2}{2s}\right)\right) dy
 \frac{1}{\sqrt{2\pi s^3}}e^{-1/2s}ds\\
&\ = \int_0^{t}
\frac 1 {2\pi s^2}
 \left( \exp\left(-\frac{(1-y)^2+1}{2s}\right)-\exp\left(-\frac{(1+y)^2+1}{2s}\right)\right) dyds.
\end{align*}
Now easy integration yields the formula stated in the lemma.
\end{proof}

\begin{lemma}\label{propo:jointDistribution}
If $W_1(0)=W_2(0)=1$ then for $y,t>0$,
\begin{align}\label{j22.1}
\P(Y_1\in dy,T_1\in dt)=\frac 1 {\pi t^2}
\left(\exp\left(-\frac{(1-y)^2+1}{2t}\right)-
\exp\left(-\frac{(1+y)^2+1}{2t}\right)\right)dtdy.
\end{align}
\end{lemma}

\begin{proof}
It follows from the definition that 
\begin{align*}
(Y_1,T_1)\stackrel{d}{=} (W_1(\tau_2),\tau_2) \1(\tau_1>\tau_2)+(W_2(\tau_1),\tau_1) \1(\tau_2>\tau_1),
\end{align*}
so for Borel sets $C$,
\begin{align*}
\P(Y_1\in C, T_1\leq t)
&= \P(W_1(\tau_2) \in C, \tau_2\leq t,  \tau_1>\tau_2)
+\P(W_2(\tau_1) \in C, \tau_1\leq t,  \tau_2>\tau_1)\\
&= 2\P(W_1(\tau_2) \in C, \tau_2\leq t,  \tau_1>\tau_2).
\end{align*}
The claim now follows from Lemma \ref{lemma:OneStoppedOneKilled}.
\end{proof}

Let $A=Y_1^{-2}$ and $B=T_1Y_1^{-2}$. Lemma \ref{propo:jointDistribution} and a standard calculation, left to the reader, show that for $a,b>0$,
\begin{align}
 &\P(A\in da, B\in db)\notag\\
 &\quad = \frac{1}{2\pi b^2\sqrt{a}}\left[\exp\left(-\frac{\left(a^{1/2}-\frac{1}{2}\right)^2+\frac{1}{4}}{b}\right)-
 \exp\left(-\frac{\left(a^{1/2}+\frac{1}{2}\right)^2+\frac{1}{4}}{b}\right)\right]\, db \, da.\label{eq:exampleDensity}
\end{align}

\begin{lemma}\label{lemma:minimumExponent}
Suppose that $\mu$ is a finite positive measure on $[a,b]$, it is absolutely continuous with respect to  Lebesgue measure, and $\mu(I)>0$ for every interval $I\subset [a,b]$ of strictly positive length. Assume that 
$f$ is a continuous function on the interval $[a,b]$.
Then 
$$\lim_{\varepsilon\to 0^+}\varepsilon\log\int_a^b e^{-f(x)/\varepsilon}\mu (dx)=-f_{\min},$$
where $f_{\min}=\inf_{x\in[a,b]}f(x)$.
\end{lemma}
\begin{proof}
For $\eps>0$,
\begin{equation}
\int_a^b e^{-f(x)/\varepsilon}\mu (dx) \leq e^{-f_{\min}/\varepsilon}\mu([a,b]). \label{eq:upperBoundMinLemma} 
\end{equation}

Suppose that $f$ attains the minimum at $x_0\in[a,b]$. For any $\delta>0$ there is an interval $I_{\delta}\subset [a,b]$ with strictly positive length,
containing $x_0$, and such that for all $x\in I_\delta$ we have $f(x)\leq f_{\min}+\delta$. Then
\begin{equation}\label{eq:lowerBoundMinLemma}
 e^{-(f_{\min}+\delta)/\varepsilon}\mu(I_{\delta})\leq \int_a^b e^{-f(x)/\varepsilon}\mu (dx).
\end{equation}
Since
\begin{align*}
\limsup_{\varepsilon\to 0^+}\varepsilon\log \mu([a,b])
= \liminf_{\varepsilon\to 0^+}\varepsilon\log \mu(I_\delta) =0,
\end{align*}
estimates \eqref{eq:upperBoundMinLemma}  and \eqref{eq:lowerBoundMinLemma} yield
$$-(f_{\min}+\delta) \leq \liminf_{\varepsilon\to 0^+}\varepsilon\log \int_a^b e^{-f(x)/\varepsilon}\mu (dx)\leq \limsup_{\varepsilon\to 0^+} \varepsilon\log\int_a^b e^{-f(x)/\varepsilon}\leq -f_{\min}.$$
The claim follows by letting $\delta\downarrow 0$.
\end{proof}

The next lemma is elementary so we leave the proof to the reader.

\begin{lemma}\label{lemma:differenceBetweenIEDFunctions}
Assume that $\lambda_1>\lambda_2\geq 0$, and $f_1$ and $f_2$ are nonnengative functions such that
$$\lim_{\varepsilon\to 0^+}\varepsilon\log f_j(\varepsilon)=-\lambda_j,$$
for $j=1,2$. Then 
$$\lim_{\varepsilon\to 0^+}\varepsilon\log (f_2(\varepsilon)\pm f_1(\varepsilon))=-\lambda_2.$$
\end{lemma}

Let  $H_1(x)=x^{-1}$.
\begin{propo}\label{propo:inverseDependencyPotentialCalculation}
  The random vector $(A,B)$ with density \eqref{eq:exampleDensity} has $(1,H_1)$-\ldm{} given by
$$g(x)=\frac{1}{2}-\frac{1}{x+2+\sqrt{4+x^2}}.$$
\end{propo}

\begin{proof}
It has been proved in \cite[Prop. 8.1]{IEDPaper} that  $g(0)=1/4$.

We will compute $g(x)$ for $x>0$.
In the following calculation we use formula \eqref{eq:exampleDensity}, and the substitution $a=u^2$ on the last line.
\begin{align}
& \P(\varepsilon Ax+B<\varepsilon)\notag \\
&=\int_0^{1/x}\int_0^{\varepsilon-\varepsilon ax}
\frac{1}{2\pi b^2\sqrt{a}}\left[\exp\left(-\frac{\left(a^{1/2}-\frac{1}{2}\right)^2+\frac{1}{4}}{b}\right)-
 \exp\left(-\frac{\left(a^{1/2}+\frac{1}{2}\right)^2+\frac{1}{4}}{b}\right)\right]\, db \, da \notag \\
&=\int_0^{1/x} \frac{1}{2\pi\sqrt{a}}
\left[\frac{\exp\left(-\frac{\left(a^{1/2}-1/2\right)^2+1/4}{\varepsilon(1- ax)}\right)}
{\left(a^{1/2}-1/2\right)^2+1/4}
-\frac{\exp\left(-\frac{\left(a^{1/2}+1/2\right)^2+1/4}{\varepsilon(1- ax)}\right)}{\left(a^{1/2}+1/2\right)^2+1/4}\right]\, da \notag \\
&=\int_0^{1/\sqrt{x}} \frac{1}{\pi}
\left[\frac{\exp\left(-\frac{\left(u-1/2\right)^2+1/4}{\varepsilon(1- u^2x)}\right)}{\left(u-1/2\right)^2+1/4}
-\frac{\exp\left(-\frac{\left(u+1/2\right)^2+1/4}{\varepsilon(1- u^2x)}\right)}{\left(u+1/2\right)^2+1/4}\right]\, du.\label{j16.3}
\end{align}
If we define measures $\mu_1$ and $\mu_2$ by
\begin{align*}
\mu_1([x_1, x_2])& =
\int_{x_1}^{x_2} \frac{1}{\pi}
\frac{1}{\left(u-1/2\right)^2+1/4}
\, du,\\
\mu_2([x_1, x_2])& =
\int_{x_1}^{x_2} \frac{1}{\pi}
\frac{1}{\left(u+1/2\right)^2+1/4}
\, du,
\end{align*}
then \eqref{j16.3} can be written as
\begin{align}\label{j16.5}
& \P(\varepsilon Ax+B<\varepsilon) \\
&=\int_0^{1/\sqrt{x}} \exp\left(-\frac{\left(u-1/2\right)^2+1/4}{\varepsilon(1- u^2x)}\right) \mu_1(du)
- \int_0^{1/\sqrt{x}} \exp\left(-\frac{\left(u+1/2\right)^2+1/4}{\varepsilon(1- u^2x)}\right) \mu_2(du).\notag
\end{align}

The function $u\mapsto \frac{\left(u-1/2\right)^2+1/4}{1- u^2x}$ attains the minimum value of
$$\frac{1}{2}-\frac{1}{x+2+\sqrt{4+x^2}},$$ at
$\frac{2}{\sqrt{4+x^2}+2+x}\in (0,1/\sqrt{x})$. 
Thus Lemma \ref{lemma:minimumExponent} implies that
\begin{align}\label{j16.4}
\lim_{\varepsilon\to 0^+}\varepsilon\log\int_0^{1/\sqrt{x}} \exp\left(-\frac{\left(u-1/2\right)^2+1/4}{\varepsilon(1- u^2x)}\right) \mu_1(du)=-\frac{1}{2}+\frac{1}{x+2+\sqrt{4+x^2}}.
\end{align}

The function $u\mapsto \frac{\left(u+1/2\right)^2+1/4}{1- u^2x}$
is  increasing  on $[0,1/x]$, so it achieves the minimum of $1/2$ at $0$.
Lemma \ref{lemma:minimumExponent} yields
\begin{align*}
\lim_{\varepsilon\to 0^+}\varepsilon\log\int_0^{1/\sqrt{x}} \exp\left(-\frac{\left(u+1/2\right)^2+1/4}{\varepsilon(1- u^2x)}\right) \mu_2(du)=-1/2.
\end{align*}
This, \eqref{j16.5}, \eqref{j16.4} and 
Lemma \ref{lemma:differenceBetweenIEDFunctions} imply that
\begin{align*}
\lim_{\varepsilon\to 0^+}\varepsilon\log \P(\varepsilon Ax+B<\varepsilon) =-\frac{1}{2}+\frac{1}{x+2+\sqrt{4+x^2}}.
\end{align*}
The proposition now follows from \eqref{eq:g(y)AsLimit}.
\end{proof}

Recall  Definitions \ref{j17.2} and \ref{j17.1}.

\begin{propo}\label{propo:phiCalculation}
 We have 
\begin{align}\label{j17.4}
\phi_1(\lambda)=
\begin{cases}
\frac{1}{4} \left(2 \sqrt{\lambda-\lambda^2}+1\right)
& \text{if  } \lambda \in[0, 1/2),\\
1/2 & \text{if  } \lambda \geq 1/2.
\end{cases}
\end{align}
The fixed point of $\phi_1$ is equal to $\lambda^\ast=1/2$.
\end{propo}
\begin{proof}
For a fixed $\lambda\in[0, 1/2)$ the function 
\begin{equation}\label{j17.3}
 x\mapsto \frac{1}{2}-\frac{1}{x+2+\sqrt{x^2+4}}+\frac{\lambda}{x}
\end{equation}
attains the minimum of $\frac{1}{4} \left(2 \sqrt{\lambda-\lambda^2}+1\right)$ at $x=\frac{4\sqrt{\lambda(1-\lambda)}}{1-2\lambda}$.
For $\lambda \geq 1/2$, the function \eqref{j17.3}
attains the minimum of $1/2$ at $x=\infty$. This proves \eqref{j17.4}.
It is easy to check that $\phi_1(1/2) =1/2$ and there are no other fixed points.
\end{proof}

\begin{lemma}\label{lemma:subgaussian}
If $X$ is an $\IED_{H_1}^1(\lambda)$-random variable with $\lambda>0$, then 
$$\lim_{t\to\infty}\frac1{t^2}\log \P(X^{-1/2}\geq t)=-\lambda.$$
\end{lemma}
\begin{proof}
See \cite[Prop. 3.6 and Example 3.8]{IEDPaper}.
\end{proof}

Let
\begin{equation*}
 (\Theta_n,\Lambda_n)
 =\left(\frac{Y_{n+1}}{Y_{n}},\frac{T_{n+1}-T_n}{Y_n^2}\right) \end{equation*}
for $n\geq 0$.
Then
$$\frac{T_{n}}{Y_{n}^2}=\frac{T_{n-1}+Y_{n-1}^2\Lambda_{n-1}}
{\Theta_{n-1}^2Y_{n-1}^2}
=\frac{1}{\Theta_{n-1}^2}\frac{T_{n-1}}{Y_{n-1}^2}+\frac{\Lambda_{n-1}}{\Theta_{n-1}^2}.$$
If we set 
\begin{align}\label{j20.10}
X_0&=0,\\
X_n&=T_{n}/Y_{n}^2,\label{j20.11} \\
A_n&=\Theta_{n-1}^{-2},\label{j20.12}\\
B_n&=\Lambda_{n-1}/\Theta_{n-1}^{2},\label{j20.13}
\end{align}
for $n\geq 1$ then
\begin{align}\label{j20.14}
X_n=A_nX_{n-1}+B_n.
\end{align}

\begin{lemma}
The sequence
$ (\Theta_n,\Lambda_n)_{n\geq 0}$
is  i.i.d.  with elements distributed as $(Y_1,T_1)$. 
The sequence $(A_n,B_n)$ is i.i.d. and its elements are
distributed as $(A,B)$ in \eqref{eq:exampleDensity}.
\end{lemma}
\begin{proof}
Recall the definition \eqref{j18.1}.
By the strong Markov property and the scaling property of Brownian motion, for every $k\geq 1$,
$$\left(\frac{\cY(T_k+tY_k^2)}{Y_k},t\geq 0\right)$$
has the same  distribution as $(\cY(t),t\geq 0)$ and is independent of $(\cY(t), t\in [0,T_k])$. Hence,
\begin{equation}
 (\Theta_n,\Lambda_n)_{n\geq 0}:=\left(\frac{Y_{n+1}}{Y_{n}},\frac{T_{n+1}-T_n}{Y_n^2}\right)_{n\geq 0} \label{eq:sequenceOfRatios}
\end{equation}
is an i.i.d. sequence with elements distributed as $(Y_1,T_1)$. 

The sequence $(A_n,B_n)_{n\geq 1}$ is i.i.d. because  $ (\Theta_n,\Lambda_n)_{n\geq 0}$
is  i.i.d.
Since $(\Theta_n,\Lambda_n) \stackrel{d} = (Y_1,T_1)$ for all $n$, it follows that
$(A_n,B_n)$ are
distributed as $(A,B)$ in \eqref{eq:exampleDensity}.
\end{proof}

\begin{lemma}\label{thm:FVSequenceConvergence}
We have
\begin{align}\label{eq:subgaussianSequenceTail}
&\lim_{t\to \infty}\frac{1}{t^2}\log \P\left(X_1^{-1/2} \geq t\right)=-1/4.
\end{align}
\end{lemma}
\begin{proof}
By Remark \ref{j27.1} random variable
$B_1$ is $\IED_{H_1}^1(\lambda_1)$, where $\lambda_1 = g(0)$. It follows from Proposition \ref{propo:inverseDependencyPotentialCalculation} that $g(0) = 1/4$ so $X_1=B_1$ is $\IED_{H_1}^1(1/4)$.  Lemma \ref{lemma:subgaussian}
now yields
\eqref{eq:subgaussianSequenceTail}.
\end{proof}

We will need the following version of the  results by Kesten \cite{Kesten} and Goldie \cite{GoldieAAP}, 
formulated in \cite[Theorem 2.4.4]{X=AX+B}.
\begin{thm}
\label{thm:KestenGoldieResult}Assume that $(A,B)$ satisfy the following conditions.
\begin{enumerate}[(i)]
 \item $A\geq 0$, a.s., and the law of $\log A$ conditioned on $\{A>0\}$ is nonarithmetic, i.e., it is not supported on 
$a\Z$ for any $a> 0$.
\item There exists $\alpha>0$ such that $\E[A^{\alpha}]=1$, $\E[|B|^{\alpha}]<\infty$ and $\E[A^{\alpha}\log^+A]<\infty$.
\item $\P(Ax+B=x)<1$ for every $x\in \R$.
\end{enumerate}
Then the equation $X\stackrel{d}{=}AX+B$ has a solution. There exist constants $c_+,c_-$
such that $c_++c_->0$ and 
\begin{align}\label{n24.3}
\P(X>x)\sim c_+x^{-\alpha}\quad\textrm{and}\quad \P(X<-x)\sim c_-x^{-\alpha},\qquad \text {  when  }x\to\infty.
\end{align}
The constants $c_+,c_-$ are given by
$$c_+=\frac{1}{\alpha m_{\alpha}}\E\left[(AX+B)_+^{\alpha}-(AX)_+^{\alpha}\right],\quad c_-=\frac{1}{\alpha m_{\alpha}}\E\left[(AX+B)_-^{\alpha}-(AX)_-^{\alpha}\right],$$
where $m_{\alpha}=\E[A^{\alpha}\log A]$.
\end{thm}

\begin{corol}\label{corol:LinearBound}
 There exists $c_1>0$ such that for all $x\geq 0$,
\begin{equation}
 \P\left(\frac{Y_n}{\sqrt{T_n}}\leq x\right)\leq c_1 x.
\end{equation}
\end{corol}
\begin{proof}
Recall \eqref{j20.10}-\eqref{j20.14}. 
Suppose that $X$ is the solution to \eqref{y20.1}. 
By Lemma \ref{pro:p1},  
\begin{align}\label{j21.1}
\P(X_n\geq x) \leq \P(X\geq x).
\end{align}

We will now verify the assumptions of Theorem \ref{thm:KestenGoldieResult}.
Assumptions (i) and (iii) clearly hold in view of \eqref{eq:exampleDensity}. 
We will show that assumption (ii) holds for $\alpha = 1/2$.

It has been proved in \cite[Prop. 8.1]{IEDPaper} that
\begin{align*}
\P(A\in da)& = \frac 4 {\pi (4 a^2 +1)} da,\qquad a>0,\\
\P(B>x) &\sim \frac 1 {\pi x} \qquad \text {  as  } x\to \infty.
\end{align*}
These formulas imply that
\begin{align*}
\E[A^{1/2}]
&=\int_0^ \infty a^{1/2}\frac 4 {\pi (4 a^2 +1)} da
=1,\\
\E[A^{1/2}\log^+A]
&=\int_0^ \infty a^{1/2} (\log ^+ a)\frac 4 {\pi (4 a^2 +1)} da
<\infty,\\
\E[|B|^{1/2}]&<\infty.
\end{align*}
The assumptions of  Theorem \ref{thm:KestenGoldieResult} are verified so we obtain
$$\P(X\geq x)\sim c_+x^{-1/2},$$
as $x\to \infty$.
This and \eqref{j21.1} give
$$\P\left(\frac{Y_n}{\sqrt{T_n}}\leq x^{-1/2}\right)
=\P\left(X_n^{-1/2}\leq x^{-1/2}\right)\leq \P\left(X^{-1/2}\leq x^{-1/2}\right)\sim c_+x^{-1/2}.$$
This implies the lemma.
\end{proof}

\begin{proof}[Proof of Theorem \ref{thm:LawOfTheIteratedLogartihm}.]
We can apply Theorem \ref{thm:LowerEnvelope} and Proposition \ref{propo:phiCalculation} to see that, a.s.,
$$\liminf_{n\to\infty}(\log n)\frac{T_n}{Y_n^2}
=\liminf_{n\to\infty}(\log n) X_n=\lambda^\ast=\frac{1}{2} .$$
Hence, 
\begin{equation}\label{eq:noIteratedLog}
\limsup_{n\to\infty}\frac{Y_n}{\sqrt{2 T_n\log n}}=1. 
\end{equation}

We will show that $\frac{\log\log T_n}{\log n}\to 1$ a.s. It follows from \eqref{eq:sequenceOfRatios} that
$Y_n=\prod_{j=1}^{n-1}\Theta_j$. It is standard to show that $\mu := \E[\log Y_1] \in (0,\infty)$ using \eqref{j22.1}.
Thus, by the Law of Large Numbers,  a.s.,
\begin{align}\label{j22.3}
\lim_{n\to \infty}\frac{\log Y_n}{n}=\E[\log Y_1]=\mu.
\end{align}
Consider any $\eps>0$.
By  Lemmas \ref{pro:p1} and \ref{thm:FVSequenceConvergence},  for large $n$,
\begin{align}\label{j22.2}
\P\left(\frac{T_n}{Y_n^2}\leq e^{-n\varepsilon}\right)
= \P\left(X_n\leq e^{-n\varepsilon}\right)
\leq \P(X_1\leq e^{-n\varepsilon/2})<\exp\left(-(1/8) e^{n\varepsilon}\right).
\end{align}
By Corollary \ref{corol:LinearBound},
$$\P\left(\frac{Y_n^2}{T_n}\leq e^{-n\varepsilon} \right)\leq c_1 e^{-n\varepsilon/2}.$$
This and \eqref{j22.2} imply that
$$\sum_{n=0}^{\infty}\P\left(\left|\frac{\log T_n -2\log Y_n}{n}\right|>\varepsilon\right)=\sum_{n=0}^{\infty}\left[\P\left(\frac{Y_n^2}{T_n}\leq e^{-n\varepsilon}\right)+\P\left(\frac{T_n}{Y_n^2}\leq e^{-n\varepsilon}\right)\right]<\infty.$$
By the Borel-Cantelli Lemma, only a finite number of events
$\left\{\left|\frac{\log T_n -2\log Y_n}{n}\right|>\varepsilon\right\}$ occur, a.s. Since this holds for every rational $\eps>0$, 
 we have $\frac{\log T_n -2\log Y_n}{n}\to 0$ a.s. We combine this observation with \eqref{j22.3} to obtain
$$\lim_{n\to\infty}\frac{\log T_n}{n}= 2\mu, \quad \text{ a.s.}$$
This implies that
$$\lim_{n\to\infty}\frac{\log\log T_n}{\log n}=1, \quad \text{ a.s.}$$
It follows from this and \eqref{eq:noIteratedLog} that, a.s.,
\begin{equation*}
\limsup_{n\to\infty}\frac{Y_n}{\sqrt{2 T_n \log\log T_n}}=1,
\end{equation*}
so the proof is complete.
\end{proof}

\section*{Acknowledgments}\rm

Research of the first author was supported in part by Simons Foundation Grant 506732. 

The third author would like to thank the Department of Mathematics at the University of Washington in Seattle, where the project took place, for the hospitality.
The third author is also grateful to the Microsoft Corporation for allowance on Azure cloud service where he ran many simulations.

\appendix

\section{}

This section is a part of the proof of Theorem \ref{thm:LowerEnvelope}. Because of the specialized nature of this material we relegated it to an appendix.

\begin{lemma}\label{propo:main:kn}
Assume that  $\tilde{\eps}, \lambda^\ast, \eps>0$  and $y_\ast>1$.
Suppose that $H$ is regularly varying at $0$ with index $-\rho<0$ and $H^{-1}$ is one of its asymptotic inverses.
There exists a strictly increasing sequence $(k_n)$ of integers such that for each $n\geq 1$,
\begin{align}\label{m14.3}
H^{-1}\left( \frac{\log k_{n+1}}{\lambda^\ast(1+\eps)} \right)y_\ast^{k_{n+1}-k_n-1}&<\tilde{\eps}, \\
H^{-1}\left( \frac{\log k_{n+1}}{\lambda^\ast(1+\eps)} \right)y_\ast^{k_{n+1}-k_n} &\geq c,\label{m14.4}
\end{align}
where $c\in(0,\tilde{\eps} y_\ast)$. 

Moreover,  for any $\gamma\in(0,1)$ there exists $K>0$ such that 
for all $n\geq 1$,
\begin{align}\label{m14.5}
k_n^\gamma\leq K n.
\end{align}
\end{lemma}

Recall that  $H^{-1}$ is regularly varying at infinity with index $-1/\rho$. 
Let  $f(x)$ be defined for $x>1$ by 
\[
f(x) = \log \tilde{\eps} - \frac{1}{\log y_\ast} \log H^{-1}\left( \frac{\log x}{\lambda^\ast(1+\eps)}\right),
\]
so that for any $k_n,k_{n+1}>1$, 
\begin{align}\label{eq1}
H^{-1}\left( \frac{\log k_{n+1}}{\lambda^\ast(1+\eps)} \right)y_\ast^{k_{n+1}-k_n}  = \tilde{\eps} y_\ast^{k_{n+1}-k_n-f(k_{n+1})}.
\end{align}

\begin{lemma}\label{m12.1}
	For any $\delta>0$, there exist $C_1,C_2\in\R$ such that
	\[
	C_1+\frac{1/\rho-\delta}{\log y_\ast}\log\log x\leq f(x)\leq C_2+\frac{1/\rho+\delta}{\log y_\ast}\log\log x
	\]
	for sufficiently large $x$.
\end{lemma}
\begin{proof}
	By regular variation of $H^{-1}$, we have for any $\delta>0$,
	\begin{align*}
	\lim_{x\to\infty}x^{1/\rho+\delta} &H^{-1}(x)  = \infty, \\	\lim_{x\to\infty}x^{1/\rho-\delta} &H^{-1}(x) = 0.
	\end{align*}
Thus, there exists $x_1>0$ such that 
\begin{align*}
x^{-1/\rho-\delta}\leq H^{-1}(x)\leq x^{-1/\rho+\delta}
\end{align*}
for $x>x_1$.  The assertion follows by using above inequalities in the definition of $f$.
\end{proof}

Suppose that $a_0>0$ is such that $f(a_0)\geq 1$ and let
\[
a_{n+1}=a_n+f(a_n),\quad n\geq 0.
\]

\begin{lemma}\label{m14.2}
The following claims hold for the sequence $(a_n)$.
    \begin{enumerate}[(i)]
      \item $a_{n+1}\geq a_n+1$ for $n\geq 0$.
      \item For any $\gamma\in(0,1)$, there exists $K>0$ such that 
      \[
      a_n^\gamma\leq K n,\quad n\geq 1.
      \]
      \item $a_{n+1}/a_n\to 1$ as $n\to\infty$.
    \end{enumerate}
\end{lemma}

\begin{proof}
(i) Since $f$ is nondecreasing, we have
\[
a_{n+1}-a_n=f(a_n)\geq f(a_0)\geq 1.
\]
(ii) We use induction and Lemma \ref{m12.1}. Fix $\gamma\in(0,1)$ and $\delta>0$. Let $C_2$ be as in Lemma \ref{m12.1}. Suppose that  $K$ is so large that, 
\begin{align}\label{m12.2}
C_2+\frac{1/\rho+\delta}{\gamma \log y_\ast} \log K +\frac{1/\rho+\delta}{\log y_\ast} \frac{1}{1-\gamma}\leq \frac{K^{1/\gamma}}{\gamma}.
\end{align}
Make $K$ larger if necessary so that  $a_1^\gamma\leq K $.

For the induction step, assume that $a_n^\gamma\leq K n$ for some $n$.
Note that $\log n \leq \frac{1}{1/\gamma-1}n^{1/\gamma-1}$.
 We use this inequality and \eqref{m12.2} to see that,
\begin{align*}
a_{n+1}& =a_n+f(a_n)\leq  (K n)^{1/\gamma}+f((K n)^{1/\gamma}) \\
& \leq (K n)^{1/\gamma}+ C_2+\frac{1/\rho+\delta}{\log y_\ast}\log\log (K n)^{1/\gamma}  \\
& \leq (K n)^{1/\gamma}+ C_2+\frac{1/\rho+\delta}{\log y_\ast}\log (K n)^{1/\gamma}  \\
&=  (K n)^{1/\gamma}+\left[C_2+\frac{1/\rho+\delta}{\gamma \log y_\ast} \log K \right]+\frac{1/\rho+\delta}{\gamma\log y_\ast} \log n \\
&\leq   (K n)^{1/\gamma}+\left[C_2+\frac{1/\rho+\delta}{\gamma \log y_\ast} \log K \right]n^{1/\gamma-1}+\frac{1/\rho+\delta}{\gamma\log y_\ast} \frac{1}{1/\gamma-1}n^{1/\gamma-1}\\
& \leq (K n)^{1/\gamma}+\frac{K^{1/\gamma}}{\gamma}n^{1/\gamma-1}
 \leq  (K (n+1))^{1/\gamma},
\end{align*}
where the last inequality follows by convexity of $x\mapsto x^{1/\gamma}$.
Part (ii) follows by induction.

(iii) By the definition of $(a_n)$ and Lemma \ref{m12.1} we have
\[
\frac{a_{n+1}}{a_n} = 1+\frac{f(a_n)}{a_n}\to 1.
\]
\end{proof}

\begin{proof}[Proof of Lemma \ref{propo:main:kn}.]
Let
\[
k_n = \lceil a_n \rceil.
\]
 Since $a_n\leq k_n< a_n+1\leq a_{n+1}\leq k_{n+1}<a_{n+1}+1$, we have
\begin{align}\label{m14.1}
y_\ast = y_{\ast}^{(a_{n+1}+1)-a_n-f(a_n)}  
>
 y_\ast^{k_{n+1}-k_n-f(k_{n+1})}  
> y_{\ast}^{a_{n+1}-(a_n+1)-f(k_{n+1})}  =  y_\ast^{-1+f(a_n)-f(k_{n+1})}.
\end{align}
By Lemma \ref{m14.2} (ii) $a_{n+1}/a_n\to 1$, by Lemma \ref{m14.2} (i) $a_n\to \infty$, and by Lemma \ref{m12.1} $ f(x) \to \infty$ as $x\to\infty$, 
so $k_{n+1}/a_n\to 1$ as $n\to \infty$. It follows that
\[
f(a_n)-f(k_{n+1})= \frac{1}{\log y^\ast} \log \frac{H^{-1}\left( \frac{\log k_{n+1}}{\lambda^\ast(1+\eps)}\right)}{H^{-1}\left( \frac{\log a_n}{\lambda^\ast(1+\eps)}\right)} \to 0, \qquad n\to \infty.
\]
Hence $y_\ast^{-1+f(a_n)-f(k_{n+1})}$, i.e.,
the right hand side of \eqref{m14.1}, converges to $1/y_\ast$ as $n\to\infty$. 
Thus, by \eqref{eq1} and \eqref{m14.1}, for large $n$, 
\begin{align*}
\tilde{\eps} y_\ast > H^{-1}\left( \frac{\log k_{n+1}}{\lambda^\ast(1+\eps)} \right)y_\ast^{k_{n+1}-k_n} \geq \frac{\tilde{\eps}}{2y_\ast}.
\end{align*}
Since  $\tilde{\eps}/(2y_\ast) < \tilde{\eps}y_\ast$, this implies \eqref{m14.3}-\eqref{m14.4}.

The bound \eqref{m14.5} follows from Lemma \ref{m14.2} (ii).
\end{proof}

\bibliographystyle{plain}
\bibliography{fixed}

\begin{thebibliography}{10}

\bibitem{BB18}
Mariusz Bieniek and Krzysztof Burdzy.
\newblock The distribution of the spine of a {F}leming-{V}iot type process.
\newblock {\em Stochastic Process. Appl.}, 128(11):3751--3777, 2018.

\bibitem{extinctionOfFlemingViot}
Mariusz Bieniek, Krzysztof Burdzy, and Soumik Pal.
\newblock Extinction of {F}leming-{V}iot-type particle systems with strong
  drift.
\newblock {\em Electron. J. Probab.}, 17:no. 11, 15, 2012.

\bibitem{regularVariation}
N.~H. Bingham, C.~M. Goldie, and J.~L. Teugels.
\newblock {\em Regular variation}, volume~27 of {\em Encyclopedia of
  Mathematics and its Applications}.
\newblock Cambridge University Press, Cambridge, 1987.

\bibitem{BKS}
V.~V. Buldygin, O.~I. Klesov, and J.~G. Steinebach.
\newblock Equivalent monotone versions of {PRV} functions.
\newblock {\em J. Math. Anal. Appl.}, 401(2):526--533, 2013.

\bibitem{X=AX+B}
Dariusz Buraczewski, Ewa Damek, and Thomas Mikosch.
\newblock {\em Stochastic models with power-law tails. The equation $X=AX+B$}.
\newblock Springer Series in Operations Research and Financial Engineering.
  Springer, 2016.

\bibitem{Gamma}
Dariusz Buraczewski, Piotr Dyszewski, Alexander Iksanov, and Alexander
  Marynych.
\newblock On perpetuities with gamma-like tails.
\newblock {\em J. Appl. Probab.}, 55(2):368--389, 2018.

\bibitem{BHM}
Krzysztof Burdzy, Robert Ho\l{}yst, and Peter March.
\newblock A {F}leming-{V}iot particle representation of the {D}irichlet
  {L}aplacian.
\newblock {\em Comm. Math. Phys.}, 214(3):679--703, 2000.

\bibitem{IEDPaper}
Krzysztof Burdzy, Bartosz Ko{\l}odziejek, and Tvrtko Tadi\'{c}.
\newblock Inverse exponential decay: stochastic fixed point equation and {ARMA}
  models.
\newblock {\em Bernoulli}, 25(4B):3939--3977, 2019.

\bibitem{Chandra}
Tapas~Kumar Chandra.
\newblock {\em The {B}orel-{C}antelli lemma}.
\newblock SpringerBriefs in Statistics. Springer, Heidelberg, 2012.

\bibitem{Breiman}
Denis Denisov and Bert Zwart.
\newblock On a theorem of {B}reiman and a class of random difference equations.
\newblock {\em J. Appl. Probab.}, 44(4):1031--1046, 2007.

\bibitem{GoldieAAP}
Charles~M. Goldie.
\newblock Implicit renewal theory and tails of solutions of random equations.
\newblock {\em Ann. Appl. Probab.}, 1(1):126--166, 1991.

\bibitem{GM00}
Charles~M. Goldie and Ross~A. Maller.
\newblock Stability of perpetuities.
\newblock {\em Ann. Probab.}, 28(3):1195--1218, 2000.

\bibitem{Grey}
D.~R. Grey.
\newblock Regular variation in the tail behaviour of solutions of random
  difference equations.
\newblock {\em Ann. Appl. Probab.}, 4(1):169--183, 1994.

\bibitem{KS}
Ioannis Karatzas and Steven~E. Shreve.
\newblock {\em Brownian motion and stochastic calculus}, volume 113 of {\em
  Graduate Texts in Mathematics}.
\newblock Springer-Verlag, New York, second edition, 1991.

\bibitem{Kesten}
Harry Kesten.
\newblock Random difference equations and renewal theory for products of random
  matrices.
\newblock {\em Acta Math.}, 131:207--248, 1973.

\bibitem{Letac}
G\'{e}rard Letac.
\newblock A contraction principle for certain {M}arkov chains and its
  applications.
\newblock In {\em Random matrices and their applications ({B}runswick, {M}aine,
  1984)}, volume~50 of {\em Contemp. Math.}, pages 263--273. Amer. Math. Soc.,
  Providence, RI, 1986.

\bibitem{Palm}
Zbigniew Palmowski and Bert Zwart.
\newblock Tail asymptotics of the supremum of a regenerative process.
\newblock {\em J. Appl. Probab.}, 44(2):349--365, 2007.

\end{thebibliography}

\end{document}